\documentclass{amsart}

\usepackage{graphicx}
\usepackage[shortlabels]{enumitem}
\usepackage{verbatim}
\usepackage{amssymb}
\usepackage{mathrsfs}  
\usepackage{mathabx}

\usepackage{color}

\usepackage{tikz}
\usetikzlibrary{arrows,chains,matrix,positioning,scopes}
\usetikzlibrary{arrows.meta}
\usetikzlibrary{matrix}

\usepackage{hyperref}
\usepackage {cleveref}

\newtheorem{theorem}{Theorem}[section]
\newtheorem{proposition}[theorem]{Proposition}
\newtheorem{lemma}[theorem]{Lemma}
\newtheorem{corollary}[theorem]{Corollary}

\theoremstyle{definition}
\newtheorem{definition}[theorem]{Definition}

\theoremstyle{remark}
\newtheorem{remark}[theorem]{Remark}
\newtheorem{example}[theorem]{Example}

\numberwithin{equation}{section}

\begin{document}

\title{The chromatic Brauer Category and its linear representations}

\author[L.F. M\"{u}ller]{L. Felipe M\"{u}ller}
\address{Mathematisches Institut, Universit\"{a}t Heidelberg, Im Neuenheimer Feld 205, 69120 Heidelberg, Germany}
\email[corresponding author]{lmueller@mathi.uni-heidelberg.de}
\author[D.J. Wrazidlo]{Dominik J. Wrazidlo}
\address{Institute of Mathematics for Industry, Kyushu University, Motooka 744, Nishi-ku, Fukuoka 819-0395, Japan}
\email{d-wrazidlo@imi.kyushu-u.ac.jp}


\date{\today.}

\keywords{Monoidal categories, Brauer category, Schauenburg tensor product, topological quantum field theories, semirings, fold maps, Kervaire spheres. \\
\quad An older version of this manuscript can be found at \url{http://export.arxiv.org/pdf/1902.05517}}

\begin{abstract}
The Brauer category is a symmetric strict monoidal category that arises as a (horizontal) categorification of the Brauer algebras in the context of Banagl's framework of positive topological field theories (TFTs).
We introduce the chromatic Brauer category as an enrichment of the Brauer category in which the morphisms are component-wise labeled.
Linear representations of the (chromatic) Brauer category are symmetric strict monoidal functors into the category of real vector spaces and linear maps equipped with the Schauenburg tensor product.
We study representation theory of the (chromatic) Brauer category, and classify all its faithful linear representations.
As an application, we use indices of fold lines to construct a refinement of Banagl's concrete positive TFT based on fold maps into the plane.
\end{abstract}

\maketitle
\tableofcontents

\section{Introduction}

The Brauer algebras $D_{m}$ have first appeared in work of Brauer \cite{bra} on representation theory of the orthogonal group $O(n)$.
In view of Schur-Weyl duality, they replace the role played by the group algebras of symmetric groups in representation theory of the general linear group.
Generators of $D_{m}$ are the diagrams consisting of $2m$ vertices and $m$ edges, where the vertices are arranged in two parallel rows of $m$ vertices, and each vertex lies in the boundary of exactly one edge.
Given a commutative ground ring $k$ with unit, $D_{m}$ is the $k(x)$-algebra freely generated as $k(x)$-module by those diagrams.
Multiplication is induced by concatenation of diagrams, where each arising free loop component gives rise to an additional multiplication with the indeterminate $x$.
A signed variant of Brauer algebras has been studied in \cite{parkam}.
Brauer algebras play an important role in knot theory, where, for instance, Birman-Murakami-Wenzl algebras \cite{birwen, mur, wen}, which are the quantized version of Brauer algebras, have been used to construct generalizations of the Jones polynomial.

We are concerned with a natural (horizontal) categorification $\textbf{Br}$ of Brauer's algebras that has been used by Banagl \cite{ban2, ban3} in search of new topological invariants in the context of his framework of positive topological field theory (TFT).
A similar category has been considered independently by Lehrer-Zhang \cite{leh} in a modern categorical approach to the invariant theory of the orthogonal and symplectic groups.
Roughly speaking, morphisms in the so-called \emph{Brauer category} $\textbf{Br}$ are represented by $1$-dimensional unoriented tangles in a high-dimensional Euclidean space.
In particular, generators and relations for the strict monoidal category $\textbf{Br}$ have been listed in \cite{ban3} (compare also \cite{leh}) by adapting the methods that are used by Turaev \cite{tur} for deriving a presentation for the category of tangle diagrams.

Let us discuss the main ideas behind Banagl's notion of positive TFT, and the role of the Brauer category $\textbf{Br}$ and its representation theory in this context.
By definition, the axioms for positive TFT \cite{ban2} differ from Atiyah's original axioms for TFT \cite{ati} in that they are formulated over semirings instead of rings.
Recall that semirings are not required to have additive inverse (``negative'') elements.
In computer science, semirings and related structures have been studied by Eilenberg \cite{eil} in the context of automata theory and formal languages.
The essential advantage of positive TFTs over usual TFTs is that so-called Eilenberg completeness of certain semirings can be used to give a rigorous construction of positive TFTs of arbitrary dimension.
This construction is implemented by Banagl in a process he calls \emph{quantization} that requires so-called fields and an action functional as input (see \Cref{general framework}).
Inspiration comes from theoretical quantum physics, where the state sum is expressed by fields and an action functional via the Feynman path integral.

In \cite{ban3}, Banagl applies his framework of quantization to produce in arbitrary dimension an explicit positive TFT for smooth manifolds.
The construction uses singularity theory of so-called fold maps, and the resulting state sum invariants can distinguish exotic smooth spheres from the standard sphere.
Now, in this concrete setting, the role of fields is played by certain fold maps into the plane, and the action functional assigns to such fields morphisms in the Brauer category $\textbf{Br}$ by extracting the $1$-dimensional patterns that arise from the singular locus of fold maps.
However, as pointed out in Section 8 of \cite{ban2}, it is desirable to compose such a category-valued action functional with a symmetric strict monoidal functor from \textbf{Br} to the category \textbf{Vect} of real vector spaces and linear maps.
Note that one requires the category \textbf{Vect} to be equipped with a symmetric \emph{strict} monoidal structure, which is provided by using the Schauenburg tensor product \cite{sch}.
In this way, the Brauer category serves only as an intermediary structure, and the state sum of the resulting positive TFT will become accessible through linear algebra.
Of course, the loss of information should be kept at a minimum during this linearization process, which is motivation for studying faithfulness of such linear representations $\mathbf{Br} \rightarrow \mathbf{Vect}$.
This knowledge is required when it comes to the explicit computation of state sum invariants (compare Section 6.3 and Section 10.5 in \cite{wra} as well as Remark 9.5 in \cite{wra2}). \\

In this paper we determine not only the faithful representations of \textbf{Br}, but also those of the \emph{chromatic Brauer category} \textbf{cBr} which will be introduced in \Cref{The chromatic Brauer category} as an enrichment of \textbf{Br} in which morphisms are component-wise labeled (``colored'') by elements of a countable index set.
Hence, in contrast to the Brauer category, isomorphic objects of the chromatic Brauer category need not be equal.
Our reason for considering \textbf{cBr} is that it can be used to construct a refinement of Banagl's positive TFT based on fold maps in the following way (see \Cref{Linearization of the fold map positive TFTs}).
In analogy with the index of non-degenerate critical points in Morse theory, one can associate a (reduced) index to the singularities of a fold map.
Those fold indices are intrinsically defined, locally constant along the singular set, and carry topological information about the source manifold.
For fold maps from $n$-dimensional source manifolds into the plane, the set of possible fold indices is $\{0, \dots, \lfloor (n-1)/2 \rfloor\}$.
We will modify Banagl's original construction by defining a \textbf{cBr}-valued action functional which additionally remembers indices of fold lines as labels from the set $\mathbb{N}=\{0,1,2,...\}$ of natural numbers.

Concerning linear representations of \textbf{Br}, Banagl has shown in Proposition 2.22 of \cite{ban3} that there exist linear representations that are faithful on loops.
This suffices for his purpose to show that state sum invariants of the positive TFT are able to detect exotic smooth structures on spheres.
As a much more general result, we have the following

\begin{theorem}[\cite{mue, wra}]\label{main1 theorem}
Let $Y \colon \mathbf{Br} \rightarrow \mathbf{Vect}$ be a symmetric strict monoidal functor from the Brauer category into the category of real vector spaces and linear maps (equipped with the Schauenburg tensor product).
Then the vector space $Y([1])$ has finite dimension $d$.
Moreover the functor $Y$ is faithful if and only if $d\geq 2$.
\end{theorem}

What is more, we will show \Cref{main theorem} below.
Since the Brauer category is naturally (monochromatically) embedded in the chromatic Brauer category, \Cref{main1 theorem} is implied by \Cref{main theorem}.
(To conclude this, one has to use that any linear representation of \textbf{Br} can be extended to one of \textbf{cBr} by means of our structure results \Cref{theorem linear representations of cBr} and the corresponding result for \textbf{Br}.)

In preparation of the statement of our result on linear representations of
\textbf{cBr}, note that the objects of \textbf{cBr} that are mapped to the
object $[1]$ of \textbf{Br} under the forgetful functor $\mathbf{cBr}
\rightarrow \mathbf{Br}$ are parametrized by the labels $k \in \mathbb{N}$,
say $([1],\underline k)$.

\begin{theorem}\label{main theorem}
Let $Y \colon \mathbf{cBr} \rightarrow \mathbf{Vect}$ be a symmetric strict monoidal functor from the chromatic Brauer category into the category of real vector spaces and linear maps (equipped with the Schauenburg tensor product).
Then, for each $k\in\mathbb{N}$, the vector space $Y(([1],\underline k))$ is of finite dimension, say $d_k$.
Suppose that $d_k > 0$ for all $k\in \mathbb{N}$.
Then, the functor $Y$ is faithful if and only if the sequence $d_{0}, d_{1}, \dots$ satisfies for all $(l_k)_{k \in \mathbb{N}} \in \bigoplus_{k=0}^{\infty}\mathbb{Z}$ the implication
\begin{align}\label{condition}
\prod_{k=0}^\infty d_k^{l_k} = 1 \qquad \Rightarrow \qquad l_k= 0
\text{ for all $k\in \mathbb{N}$}.
\end{align}
\end{theorem}
In particular, faithful linear representations of \textbf{cBr} exist because one can take $d$ to be the sequence of prime numbers, and then apply \Cref{theorem linear representations of cBr} to construct a strict monoidal functor $Y \colon \mathbf{cBr} \rightarrow \mathbf{Vect}$ which realizes $d_k$ for $k \in \mathbb{N}$ as the dimension of the real vector space $Y(([1],\underline{k}))$.
We note, however, that the condition that implication (\ref{condition}) holds for all $(l_k)_{k \in \mathbb{N}} \in \bigoplus_{k=0}^{\infty}\mathbb{Z}$ is not equivalent to saying that every finite subset of the sequence $d_{0}, d_{1}, \dots$ is relatively prime (e.g., take $d_{0}, d_{1}, \dots$ to be the double $6, 10, 14, \dots$ of the sequence of odd prime numbers).

The paper is structured as follows.
In \Cref{Preliminaries on strict monoidal categories} we recall fundamental facts about monoidal categories in general, and the Schauenburg tensor product in particular.
The chromatic Brauer category is introduced in \Cref{The chromatic Brauer category and its linear representations}, where its linear representations are classified by \Cref{theorem linear representations of cBr}.
The proof of our main result \Cref{main theorem} will be given in \Cref{proof of main theorem}.
Finally, in \Cref{Linearization of the fold map positive TFTs} we discuss our application to Banagl's positive TFT based on fold maps.

\subsection*{Notation}
Throughout the paper, the natural numbers will be meant to be the set $\mathbb{N}=\{0,1,2,\dots\}$ (including zero).

\section{Preliminaries on strict monoidal categories}\label{Preliminaries on strict monoidal categories}

In this section, we introduce the definitions and notational conventions on strict monoidal categories (\Cref{subsection monoidal category}) and the Schauenburg tensor product (\Cref{The Schauenburg tensor product}) that will be used throughout the paper.

\subsection{Strict monoidal categories}\label{subsection monoidal category}
We refer to Kassel \cite{kas} for the basic definitions that are recalled in this section.

A \textit{monoidal category} $(\mathbf C,\otimes,I,\alpha,\lambda,\rho)$ is a
category $\mathbf C$ equipped with a bifunctor
$\otimes\colon \mathbf C\times \mathbf C\rightarrow \mathbf C$,
an object $I\in\operatorname{Ob}(\mathbf C)$, called \textit{unit} with
respect to the tensor product $\otimes$, and three isomorphisms
\begin{align*}
	\alpha_{X,Y,Z}\colon(X\otimes Y)\otimes Z\rightarrow
	X\otimes (Y\otimes Z),\,
	\lambda_X\colon I\otimes X\rightarrow X\text{ and }
	\rho_Y\colon Y\otimes I\rightarrow Y,
\end{align*}
which are functorial in $X,Y,Z\in \operatorname{Ob}(\mathbf C)$.
Furthermore, $\alpha$, $\lambda$ and $\rho$ satisfy the 
\textit{coherence conditions} given by the \textit{pentagon axiom}
and \textit{triangle axiom}. These are
	\small
	\begin{align*}
	\begin{tikzpicture}[->,>=stealth,shorten >=1pt,auto,semithick, scale=.95]
		\node (a) at (0,0) {$((W\otimes X)\otimes Y)\otimes Z$};
		\node (b) at (4.5,0) {$(W\otimes (X\otimes Y))\otimes Z$};
		\node (c) at (9,0) {$W\otimes ((X\otimes Y)\otimes Z)$};
		\node (d) at (9,-2) {$W\otimes (X\otimes (Y\otimes Z))$};
		\node (e) at (0,-2) {$(W\otimes X)\otimes (Y\otimes Z)$};
		\path (a) edge node {$\alpha_{W,X,Y}\otimes 1_Z$}(b)
				edge node [swap]{$\alpha_{W\otimes X,Y,Z}$} (e)
			(b) edge node {$\alpha_{W,X\otimes Y,Z}$} (c)
			(c) edge node {$1_W\otimes \alpha_{X,Y,Z}$}(d)
			(e) edge node [swap]{$\alpha_{W,X,Y\otimes Z}$} (d);
	\end{tikzpicture}
	\end{align*}
	\normalsize
and
	\begin{align*}
		\begin{tikzpicture}[->,>=stealth,shorten >=1pt,auto,semithick]
		\node (a) at (0,0) {$(Y\otimes I)\otimes X$};
		\node (b) at (4,0) {$Y\otimes X$};
		\node (c) at (2,-2) {$Y\otimes (I\otimes X)$};
		\path (a) edge node {$\rho_Y\otimes 1_X$}(b)
				edge node [swap]{$\alpha_{Y,I,X}$} (c)
			(c) edge node [swap]{$1_Y\otimes \lambda_X$} (b);
	\end{tikzpicture},
	\end{align*}
for all $W,X,Y,Z\in\operatorname{Ob}(\mathbf C)$.
Here, $\alpha$ is called \textit{associativity constraint}, and $\lambda$
and $\rho$ are called \textit{left} and \textit{right unit constraints},
respectively.
A monoidal category $\mathbf{C}$ is called \textit{strict} if the associativity and unit constraints $\alpha,\lambda,\rho$ are given by identity morphisms of the category.

If $(\mathbf C,\otimes,I,\alpha,\lambda,\rho)$ is a (strict) monoidal category, then we will call the data $(\otimes, I,\alpha,\lambda,\rho)$ a \textit{(strict) monoidal structure} on $\mathbf C$.
In most of the paper, we work with strict monoidal categories.
In that case, we will omit the associativity and unit constraints in the notation
of a monoidal category, i.e. we will write $(\mathbf C,\otimes, I)$ instead of
$(\mathbf C,\otimes,I,\alpha,\lambda,\rho)$.

Let $\mathbf C$ and $\mathbf D$ be monoidal categories. A \textit{monoidal functor} is a functor $F:\mathbf C\rightarrow\mathbf D$ which respects the monoidal structure. To be more precise, it is a functor
\begin{align*}
	(F,\xi,\xi_0)\colon (\mathbf C,\otimes_{\mathbf C},
	I_{\mathbf C},\alpha_{\mathbf C},\lambda_{\mathbf C},\rho_{\mathbf C})
	\rightarrow (\mathbf D,\otimes_{\mathbf D},I_{\mathbf D},
	\alpha_{\mathbf D},\lambda_{\mathbf D},\rho_{\mathbf D}),
\end{align*}
with isomorphisms $\xi_{X,Y}\colon F(X)\otimes_{\mathbf D} F(Y)\rightarrow
F(X\otimes_{\mathbf C} Y)$ for all $X,Y\in\operatorname{Ob}(\mathbf C)$,
functorial in $X,Y$, i. e. for morphisms $f\colon X\rightarrow X'$ and $g\colon Y\rightarrow Y'$ the diagram
\begin{align*}
	\begin{tikzpicture}[->,>=stealth,shorten >=1pt,auto,semithick,scale=.77]
		\node (a) at (0,0) {$F(X)\otimes_{\mathbf D}F(Y)$};
		\node (b) at (5,0) {$F(X\otimes_{\mathbf C} Y)$};
		\node (c) at (0,-2.5) {$F(X')\otimes_{\mathbf D} F(Y')$};
		\node (d) at (5,-2.5) {$F(X'\otimes_{\mathbf C}Y')$};
		\path (a) edge node {$\xi_{F(X),F(Y)}$} (b)
			(b) edge node {$F(f\otimes_{\mathbf{C}}f')$} (d)
			(a) edge node[swap] {$F(f)\otimes_{\mathbf{D}} F(f')$} (c)
			(c) edge node[swap] {$\xi_{F(X'),F(Y')}$} (d);
	\end{tikzpicture}
\end{align*}
 and an isomorphism $\xi_0\colon
I_{\mathbf D}\rightarrow F(I_{\mathbf C})$ such that the diagrams
\small
\begin{align*}
	\begin{tikzpicture}[->,>=stealth,shorten >=1pt,auto,semithick,scale=.77]
		\node (a) at (0,0) {$F(X)\otimes_{\mathbf D}
			(F(Y)\otimes_{\mathbf D} F(Z))$};
		\node (b) at (5.5,0) {$F(X)\otimes_{\mathbf D}
			F(Y\otimes_{\mathbf C} Z)$};
		\node (c) at (10.5,0) {$F(X\otimes_{\mathbf C}
			(Y\otimes _{\mathbf C} Z))$};
		\node (d) at (0,2) {$(F(X)\otimes_{\mathbf D}F(Y))
			\otimes_{\mathbf D} F(Z)$};
		\node (e) at (5.5,2) {$F(X\otimes_{\mathbf C} Y)
			\otimes_{\mathbf D} F(Z)$};
		\node (f) at (10.5,2){$F((X\otimes_{\mathbf C} Y)
			\otimes_{\mathbf C} Z)$};
		\path (d) edge node[swap] 
			{${\alpha_{\mathbf D}}_{F(X),F(Y),F(Z)}$} (a)
			edge node {$\xi_{X,Y}\otimes_{\mathbf D} 1_{F(Z)}$} (e)
			(a) edge node[swap] {$1_{F(X)}
			\otimes_{\mathbf D} \xi_{Y,Z}$} (b)
			(b) edge node [swap]{$\xi_{X,Y\otimes_{\mathbf C} Z}$} (c)
			(e) edge node {$\xi_{X\otimes_{\mathbf C} Y,Z}$} (f)
			(f) edge node {$F({\alpha_{\mathbf C}}_{X,Y,Z})$} (c);
		\node (a) at (-1,-4) {$F(X)$};
		\node (b) at (3,-4) {$F(I_{\mathbf C}\otimes_{\mathbf C} X)$};
		\node (c) at (-1,-2) {$I_{\mathbf D}\otimes_{\mathbf D} F(X)$};
		\node (d) at (3,-2){$F(I_{\mathbf C})\otimes_{\mathbf D} F(X)$};
		\path (c) edge node {$\xi_0\otimes_{\mathbf D} 1_{F(X)}$} (d)
			(d) edge node {$\xi_{I_{\mathbf C},X}$} (b)
			(b) edge node {$F({\lambda_{\mathbf C}}_X)$} (a)
			(c) edge node [swap]{${\lambda_{\mathbf D}}_{F(X)}$} (a);
		\node (aa) at (7.5,-4) {$F(X)$};
		\node (bb) at (11.5,-4) {$F(X\otimes_{\mathbf C} I_{\mathbf C})$};
		\node (cc) at (7.5,-2) {$F(X)\otimes_{\mathbf D} I_{\mathbf D}$};
		\node (dd) at (11.5,-2){$F(X)\otimes_{\mathbf D} F(I_{\mathbf 
			C})$};
		\path (cc) edge node {$1_{F(X)}\otimes_{\mathbf D} \xi_0$} (dd)
			(dd) edge node {$\xi_{X,I_{\mathbf C}}$} (bb)
			(bb) edge node {$F({\rho_{\mathbf C}}_X)$} (aa)
			(cc) edge node [swap]{${\rho_{\mathbf D}}_{F(X)}$} (aa);
	\end{tikzpicture}
\end{align*}
\normalsize
commute for all $X,Y,Z\in\operatorname{Ob}(\mathbf{C})$.
The monoidal functor $(F,\xi,\xi_0)$ is called \textit{strict} if the
isomorphisms $\xi_0$ and $\xi$ are identity morphisms of $\mathbf D$.
\\

Let $(\mathbf{C},\otimes,I,\alpha,\lambda,\rho)$ be a monoidal category.
A \textit{symmetric braiding} $b$ on $\mathbf C$ is a monoidal functorial
isomorphism $b\colon \otimes\rightarrow \otimes \circ \tau$,
where the map $\tau\colon\operatorname{Ob}(\mathbf C)\times
\operatorname{Ob}(\mathbf C)\rightarrow
\operatorname{Ob}(\mathbf C)\times \operatorname{Ob}(\mathbf C)$
is given by $\tau(X,Y)=(Y,X)$ for any pair $(X,Y)$ of objects of the 
category $\mathbf C$, satisfying the \textit{hexagon axiom}, the \textit{unity
coherence}, and the \textit{inverse law} which are given by the commutativity
of the diagrams
\begin{align*}
	\begin{tikzpicture}[->,>=stealth,shorten >=1pt,auto,semithick]
		\node (a) at (3,0) {$(X\otimes Y)\otimes Z$};
		\node (b) at (7,0) {$X\otimes(Y\otimes Z)$};
		\node (c) at (7,-2) {$(Y\otimes Z)\otimes X$};
		\node (d) at (7,-4){$Y\otimes( Z\otimes X)$};
		\node (e) at (3,-2){$(Y\otimes X)\otimes Z$};
		\node (f) at (3,-4){$Y\otimes(X\otimes Z)$};
		\path (a) edge node {$\alpha_{X,Y,Z}$} (b)
			(b) edge node {$b_{X,Y\otimes Z}$} (c)
			(c) edge node {$\alpha_{Y,Z,X}$} (d)
			(a) edge node [swap] {$b_{X,Y}\otimes 1_Z$} (e)
			(e) edge node [swap] {$\alpha_{Y,X,Z}$} (f)
			(f) edge node [swap] {$1_Y\otimes b_{X,Z}$} (d);
	\node (a) at (0,-5.25) {$X\otimes I$};
	\node (b) at (4,-5.25) {$I\otimes X$};
	\node (c) at (2,-7.25) {$X$};
	\path (a) edge node {$b_{X,I}$} (b)
		(b) edge node {$\lambda_X$} (c)
		(a) edge node [swap]{$\rho_X$} (c);
	\node (a) at (6,-5.25) {$X\otimes Y$};
	\node (b) at (10,-5.25) {$X\otimes Y$};
	\node (c) at (8,-7.25) {$Y\otimes X$};
	\path (a) edge node {$1_{X\otimes Y}$} (b)
		(a) edge node [swap]{$b_{X,Y}$} (c)
		(c) edge node [swap]{$b_{Y,X}$} (b);
\end{tikzpicture},
\end{align*}
respectively,  for all $X,Y,Z\in \operatorname{Ob}(\mathbf C)$.
A monoidal category $\mathbf C$ together with a symmetric braiding $b$ is called \textit{symmetric monoidal category}.

Let $(\mathbf C,\otimes,I_{\mathbf C},b_{\mathbf C})$ and $(\mathbf D,
\odot,I_{\mathbf D},b_{\mathbf D})$ be symmetric strict monoidal
categories. A monoidal functor $(F,\xi,\xi_0)\colon
\mathbf C\rightarrow\mathbf D$ is called \textit{symmetric} if
$F$ is compatible with the symmetric structures, i.e. if
$F({b_{\mathbf C}}_{X,Y})=\xi_{F(Y),F(X)}\circ {b_{\mathbf{D}}}_{F(X),F(Y)}
\circ \xi_{X,Y}^{-1}$ for all $X,Y\in \operatorname{Ob}(\mathbf C)$.
\\

\subsection{The Schauenburg tensor product}\label{The Schauenburg tensor product}
The  usual tensor product $\otimes$ in the category $\mathbf{Vect}$ of real vector spaces and linear maps determines a monoidal structure on $\mathbf{Vect}$ (with unit object $I=\mathbb{R}$) which is certainly not strict.
When studying linear representations of strict monoidal categories, it is appropriate to endow $\mathbf{Vect}$ with a strict monoidal structure.
As described in Theorem XI.5.3 in \cite[p. 291]{kas}, there is a general method for turning any given monoidal category $\mathbf C$ into a monoidally equivalent strict monoidal category $\mathbf C^{\operatorname{str}}$.
However, this procedure changes the category even on the object level, which is probably not convenient for studying properties of linear representations.
Banagl employs an explicit strict monoidal structure on $\mathbf{Vect}$, namely the \textit{Schauenburg tensor product} $\odot$ introduced in \cite{sch}.
The resulting strict monoidal category $(\mathbf{Vect}, \odot, I)$ is monoidally equivalent to the usual monoidal structure on $\mathbf{Vect}$ obtained by $\otimes$.
The advantage of Schauenburg's construction is that the monoidal equivalence $(F,\xi,\xi_0)$ can be chosen in such a way that $F$ is the identity on $\mathbf{Vect}$, and $\xi_0$ is the identity on $I$ (see Theorem 4.3 in \cite{sch}).
Via the natural isomorphism $\xi\colon\otimes\rightarrow\odot$ we can define elements $v\odot w\in V\odot W$, $v\odot w:=\xi_{V,W}(v\otimes w)$, for $v\in V$ and $w\in W$. Given elements $u\in U$, $v\in V$ and $w\in W$, the identity $(u\odot v)\odot w=u\odot (v\odot w)$ holds because of $(U\odot V)\odot W= U\odot (V\odot W)$.

There is a standard symmetric braiding $\beta_{V,W}\colon V\otimes W \rightarrow W\otimes V$ on $\mathbf{Vect}$ for all $V,W\in\operatorname{Ob}(\mathbf{Vect})$ (with respect to the standard tensor product $\otimes$).
By defining $b_{V,W}=\xi_{W,V}\circ \beta_{V,W}\circ \xi_{V,W}^{-1}$ we obtain a symmetric braiding with respect to $\odot$.
All in all, the following proposition holds.

\begin{proposition}
	The data $(\mathbf{Vect},\odot,\mathbb{R},b)$ define a symmetric
	strict monoidal category.
\end{proposition}

For the rest of this paper, we will use the Schauenburg tensor product on $\mathbf{Vect}$; thus, we will from now on write $\otimes$ instead of $\odot$.

\section{The chromatic Brauer category and its linear representations}\label{The chromatic Brauer category and its linear representations}

In this section, we first provide some background on compact closed categories (see \Cref{Compact closed categories}), and then introduce the chromatic Brauer category as the certain quotient of a free strict compact closed category (see \Cref{The chromatic Brauer category}).

\subsection{Compact closed categories}\label{Compact closed categories}
Let $(\mathbf{C},\otimes,I,\alpha,\lambda,\rho, b)$ be a symmetric monoidal category.
An object $X\in \operatorname{Ob}(\mathbf{C})$ is called \textit{dualizable} if there exists a triple $(X^\star, i_X,e_X)$ consisting of an object $X^\star\in \operatorname{Ob}(\mathbf{C})$, called a \textit{dual} of $X$, and morphisms
$i_X\colon I\rightarrow X\otimes X^\star$ and $e_X\colon X^\star\otimes X\rightarrow I$, called \textit{unit} and \textit{counit} respectively, such that the ``triangular equations''
\begin{align*}
\begin{tikzpicture}[->,>=stealth,shorten >=1pt,auto,semithick]
	\node (a) at (0,0) {$X^\star\otimes I$};
	\node (b) at (0,-3){$X^\star$};
	\node (c) at (3.5,0){$X^\star\otimes (X\otimes X^\star)$};
	\node (d) at (3.5,-1.5){$(X^\star\otimes X)\otimes X^\star$};
	\node (e) at (3.5,-3){$I\otimes X^\star$};
	\path (a) edge node [swap]{$\rho_{X^\star}$} (b)
			edge node {$1_{X^\star}\otimes i_{X^\star}$} (c)
		(c) edge node {$\alpha_{X^\star,X,X^\star}^{-1}$} (d)
		(d) edge node {$e_{X^\star}\otimes 1_{X^\star}$} (e)
		(e) edge node{$\lambda_{X^\star}$} (b);
	\node (a) at (6.5,0) {$I\otimes X$};
	\node (b) at (6.5,-3){$X$};
	\node (c) at (10,0){$(X\otimes X^\star)\otimes X$};
	\node (d) at (10,-1.5){$X\otimes(X^\star\otimes X)$};
	\node (e) at (10,-3){$X\otimes I$};
	\path (a) edge node [swap]{$\lambda_{X}$} (b)
			edge node {$i_X\otimes 1_{X}$} (c)
		(c) edge node {$\alpha_{X,X^\star,X}$} (d)
		(d) edge node {$1_{X}\otimes e_{X}$} (e)
		(e) edge node  {$\rho_{X}$} (b);
\end{tikzpicture}
\end{align*}
hold.
A \emph{compact closed category} is a symmetric monoidal category in which every object is dualizable.
We use the notation $(\mathbf{C},\otimes,I,\alpha,\lambda,\rho, b, ()^{\star}, i, e)$ to include as data the assignment $()^{\star} \colon X \mapsto X^{\star}$, as well as the families $i$ and $e$ of unit morphisms and counit morphisms, respectively.
As explained in Section 6 of \cite{kel80}, the axioms of a compact closed category determine three canonical isomorphism $u_{X, Y} \colon (X \otimes Y)^{\star} \rightarrow Y^{\star} \otimes X^{\star}$, $v \colon I^{\star} \rightarrow I$, and $w_{X} \colon X^{\star \star} \rightarrow X$ for all objects $X, Y$ in $\mathbf{C}$ which are uniquely determined by certain commutative diagrams.
Following Section 9 of \cite{kel80}, We call a compact closed category \emph{monoidally strict} if the underlying monoidal category is strict.
Furthermore, we call a compact closed category \emph{strict} if it is monoidally strict, and, in addition, the families of isomorphisms $u$, $v$, $w$ described above are all identity morphisms.

In \cite{kel80}, Kelly and Laplaza give an explicit description of the free compact closed category $F \mathbf{A}$ on a given category $\mathbf{A}$, thus solving the coherence problem for this structure.
In Section 9 of \cite{kel80}, they discuss some related structures including the free monoidally strict compact closed category $F' \mathbf{A}$, as well as the free strict compact closed category $F'' \mathbf{A}$ on $\mathbf{A}$.
The latter category is relevant for the present paper, and has the following description in terms of generators and relations.

\begin{theorem}[\cite{kel80}]\label{free strict compact closed category}
For any category $\mathbf{A}$, the free strict compact closed category $F'' \mathbf{A}$ has as objects the tensor products $X_{1} \otimes \dots \otimes X_{n}$, $n \geq 0$, where each $X_{i}$ is $A$ or $A^{\star}$ for some object $A$ of $\mathbf{A}$.
Moreover, $F'' \mathbf{A}$ is a strict compact closed category which is generated (as strict monoidal category) by the families of morphisms of its symmetric braiding $b$, its unit $i$, its counit $e$, and the morphisms of $\mathbf{A}$ (considered as morphisms in $F'' \mathbf{A}$), and by the following relations:
\begin{description}
\item[(R1) Naturality of $b$]
For all morphisms $f \colon X \rightarrow Y$ and all objects $Z$ in $F'' \mathbf{A}$,
\begin{align*}
b_{Y, Z} \circ (f \otimes 1_{Z}) = (1_{Z} \otimes f) \circ b_{X, Z}, \\
b_{Z, Y} \circ (1_{Z} \otimes f) = (f \otimes 1_{Z}) \circ b_{Z, X}.
\end{align*}
\item[(R2) Coherence relations for $b$]
For all objects $X, Y, Z$ in $F'' \mathbf{A}$,
\begin{align*}
b_{X, Y \otimes Z} = (1_{Y} \otimes b_{X, Z}) \circ (b_{X, Y} \otimes 1_{Z}) \qquad \textit{(hexagon axiom)}, \\
b_{X, I} = 1_{X} \qquad \textit{(unity coherence)}, \\
b_{Y, X} \circ b_{X, Y} = 1_{X \otimes Y} \qquad \textit{(inverse law)}.
\end{align*}
\item[(R3) Triangular equations for $i$ and $e$]
For all objects $X$ in $F'' \mathbf{A}$,
\begin{align*}
(1_{X^\star}\otimes i_{X^\star}) \circ (e_{X^\star}\otimes 1_{X^\star}) = 1_{X^\star}, \\
(1_{X}\otimes e_{X}) \circ (i_{X}\otimes 1_{X}) = 1_{X}.
\end{align*}
\item[(R4) Functoriality of the canonical assignment $\{\} \colon \mathbf{A} \rightarrow F'' \mathbf{A}$]
For all morphisms $g \colon A \rightarrow B$, $h \colon B \rightarrow C$, and all objects $D$ in $\mathbf{A}$,
\begin{align*}
\{h \circ g\} = \{h\} \circ \{g\}, \\
1_{D} = \{1_{D}\}.
\end{align*}
\item[(R5) Strictness]
For all objects $X, Y$ in $F'' \mathbf{A}$,
\begin{align*}
(1_{X} \otimes i_{Y} \otimes 1_{X^\star}) \circ i_{X} = i_{X \otimes Y} \qquad \textit{(see diagram (6.2) in \cite{kel80})}, \\
i_{I} = 1_{I} \qquad \textit{(see diagram (6.3) in \cite{kel80})}, \\
b_{X, X^\star} \circ i_{X} = i_{X^\star} \qquad \textit{(see diagram (6.4) in \cite{kel80})}.
\end{align*}
\end{description}
\end{theorem}

\subsection{The chromatic Brauer category} \label{The chromatic Brauer category}

The Brauer category $\mathbf{Br}$ is considered in Section 10 of \cite{ban2} and Section 2.3 of \cite{ban3}, where Banagl defines it as a natural (horizontal) categorification of Brauer algebras.
Roughly speaking, isomorphism classes of finite sets serve as objects, and morphisms are isotopy classes of unoriented tangles in Euclidean $4$-space.
We shall next introduce the \textit{chromatic Brauer category} $\mathbf{cBr}$ as a certain enrichment of the Brauer category.
Namely, we equip the components of objects and morphisms with colorings using a countable number of colors (see \Cref{remark cbr intuition}).
However, we shal first give a more sophisticated definition of $\mathbf{cBr}$ in terms of the free strict compact closed category assigned to the discrete category $\mathbf{N}$ with $\operatorname{Ob}(\mathbf{N}) = \{(k); \; k \in \mathbb{N}\}$ a countable set indexed by the natural numbers.

\begin{definition}[chromatic Brauer category]\label{definition cbr}
The category $\mathbf{cBr}$ is the quotient of the free strict compact closed category $F'' \mathbf{N}$ (see \Cref{free strict compact closed category}) by the relation $(k) = (k)^\star$ for all objects $(k) \in \operatorname{Ob}(\mathbf{N})$.
In other words, the objects of the strict compact closed category $\mathbf{cBr}$ are of the form
\begin{equation}\label{objects of cbr}
(k_{1}) \otimes \dots \otimes (k_{m}), \quad m \geq 0 \text{ and } k_{i} \in \mathbb{N} \text{ for } 1 \leq i \leq m,
\end{equation}
and $\mathbf{cBr}$ is generated (as strict monoidal category) by the families of morphisms of its symmetric braiding $b$, its unit $i$, and its counit $e$, and by the following relations:
\begin{description}
\item[(C1) Naturality of $b$]
For all morphisms $f \colon X \rightarrow Y$ and all objects $Z$ in $\mathbf{cBr}$,
\begin{align*}
b_{Y, Z} \circ (f \otimes 1_{Z}) = (1_{Z} \otimes f) \circ b_{X, Z}, \\
b_{Z, Y} \circ (1_{Z} \otimes f) = (f \otimes 1_{Z}) \circ b_{Z, X}.
\end{align*}
\item[(C2) Coherence relations for $b$]
For all objects $X, Y, Z$ in $\mathbf{cBr}$,
\begin{align*}
b_{X, Y \otimes Z} = (1_{Y} \otimes b_{X, Z}) \circ (b_{X, Y} \otimes 1_{Z}) \qquad \textit{(hexagon axiom)}, \\
b_{X, I} = 1_{X} \qquad \textit{(unity coherence)}, \\
b_{Y, X} \circ b_{X, Y} = 1_{X \otimes Y} \qquad \textit{(inverse law)}.
\end{align*}
\item[(C3) Triangular equations for $i$ and $e$]
For all objects $X$ in $\mathbf{cBr}$,
\begin{align*}
(1_{X^\star}\otimes i_{X^\star}) \circ (e_{X^\star}\otimes 1_{X^\star}) = 1_{X^\star}, \\
(1_{X}\otimes e_{X}) \circ (i_{X}\otimes 1_{X}) = 1_{X}.
\end{align*}
\item[(C4) Strictness]
For all objects $X, Y$ in $\mathbf{cBr}$,
\begin{align*}
 i_{X \otimes Y} = (1_{X} \otimes i_{Y} \otimes 1_{X^\star}) \circ i_{X}, \\
i_{I} = 1_{I}, \\
b_{X, X^\star} \circ i_{X} = i_{X^\star}.
\end{align*}
\end{description}
\end{definition}

For an object in $\operatorname{Ob}(\mathbf{cBr})$ of the form (\ref{objects of cbr}), we introduce the alternative notation $([m], c)$, where $[m]$ denotes the set $\{1, \dots, m\}$ (with $[0] = \emptyset$), and $c \colon [m] \rightarrow \mathbb{N}$ is the map given by $c(i) = k_{i}$.
If $c(i) = k$ for all $i$, then we write $c = \underline{k}$.
In later sections, we will mainly use the notation $([m], c)$ for objects in $\mathbf{cBr}$ because of its consistency with the notation $[m]$ used for objects in the Brauer category $\mathbf{Br}$ in \cite{ban2, ban3}.

The \emph{elementary morphisms} $b_{(k), (l)}, i_{(k)}$, and $e_{(k)}$, $k, l \in \mathbb{N}$, play an important role in the study of $\mathbf{cBr}$, as the following remarks show.

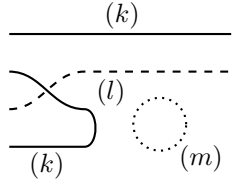
\begin{figure}
\begin{center}
\begin{tikzpicture}[auto,scale=.5,thick]
				\draw[dashed] (0,1) to [out=0,in=180] ++(2,1)
					to [out=0,in=180] ++(4,0);
				\draw[] (0,2) to [out=0,in=180] ++(2,-1)
					to [out=0,in=0] ++(0,-1)
					to [out=180,in=0] ++(-2,0);
				\draw[dotted] (4,0.6) circle (20pt);
				\draw (0,3) to (5.9,3);
				\node (A) at (5.1,-0.3) {$(m)$};
				\node (B) at (3,3.5) {$(k)$};
				\node (C) at (2.7,1.45) {$(l)$};
				\node (D) at (1,-0.5) {$(k)$};
		\end{tikzpicture}
\end{center}
\caption{A morphism $(k) \otimes (l) \otimes (k) \otimes (k) \rightarrow (l) \otimes (k)$ in $\mathbf{cBr}$.}
\label{example morphism cbr}
\end{figure}

\begin{remark}[visualizing morphisms of $\mathbf{cBr}$]\label{remark cbr intuition}
Similar to the discussion of morphisms of the Brauer category considered in \cite{ban2, ban3}, morphisms of $\mathbf{cBr}$ can be thought of as being ambient isotopy classes of $1$-dimensional unoriented tangles in a high dimensional Euclidean space with the novelty that each component is now labeled by an object of $\mathbf{N}$.
By projecting the tangles into a plane, we can visualize morphisms of $\mathbf{cBr}$ by deformation classes of diagrams like in \Cref{example morphism cbr}, where differently structured lines correspond to independent labels.
We can use these diagrams to encode a decomposition of a morphism of $\mathbf{cBr}$ into elementary morphisms as follows.
Horizontal lines in such a diagram represent identity morphisms, crossings represent elementary braidings $b_{(k), (l)}$, and left and right half circles represent elementary units $i_{(k)}$ and counits $e_{(k)}$, respectively.
Moreover, composition of morphisms in $\mathbf{cBr}$ corresponds to horizontal composition of diagrams, and tensor product of morphisms in $\mathbf{cBr}$ is defined by vertical ``stacking'' of diagrams.
\end{remark}

\begin{remark}[generators and relations of $\mathbf{cBr}$]\label{generators and relations of cbr}
It follows from the relations $b_{X, Y \otimes Z} = (1_{Y} \otimes b_{X, Z}) \circ (b_{X, Y} \otimes 1_{Z})$ \textbf{(C2)} and $i_{X \otimes Y} = (1_{X} \otimes i_{Y} \otimes 1_{X^\star}) \circ i_{X}$ \textbf{(C4)} (and the analogous relation for the counit $e$) that we can take the elementary morphisms $b_{(k), (l)}, i_{(k)}$, and $e_{(k)}$ for $k, l \in \mathbb{N}$ as generators of $\mathbf{cBr}$.
Then, it can be shown that the following relations are sufficient to generate $\mathbf{cBr}$ as strict monoidal category (also compare with Theorem 2.6 in \cite{leh}):
	\begin{description}
		\item[(A1) Zig-Zag (straightening)]
			\begin{align*}
			\hspace{-4em}(e_{(k)}\otimes 1_{(k)})\circ(1_{(k)}
			\otimes i_{(k)})&=1_{(k)},
			\\\vspace{1em}
			&\hspace{-6.5em}\begin{tikzpicture}[scale=.5, thick]
				\draw[] (0,0) to [out=0,in=180](2,0);
				\draw (2,0) to [out=0,in=0] (2,1);
				\draw (2,1) to [out=180,in=180] (2,2);
				\draw (2,2) to [out=0,in=180](4,2);
				\node  (a) at (5,1) {$=$};
				\draw (6,1) to [out=0,in=180] (8,1);
			\end{tikzpicture}\quad.
			\end{align*}
		\item[$\mathbf{(A2)}$ Sliding]
			\begin{align*}
				(e_{(k)}\otimes 1_{(l)})\circ(1_{(k)}\otimes b_{(l), (k)})
				&= (1_{(l)}\otimes
				e_{(k)})\circ(b_{(k), (l)}\otimes 1_{(k)}), \\
				&\hspace{-6.5em}\begin{tikzpicture}[auto,scale=.5,thick]
				\draw[dashed] (0,1) to [out=0,in=180] ++(2,1)
					to [out=0,in=180] ++(2,0);
				\draw[] (0,2) to [out=0,in=180] ++(2,-1)
					to [out=0,in=0] ++(0,-1)
					to [out=180,in=0] ++(-2,0);
				\node (a) at (5,1) {$=$};
				\draw[] (6,0) to [out=0,in=180] ++(2,1)
					to [out=0,in=0] ++(0,1)
					to [out=180,in=0] ++(-2,0);
				\draw[dashed] (6,1) to [out=0,in=180] ++(2,-1)
					to [out=0,in=180] ++(2,0);
		\end{tikzpicture}\quad.
	\end{align*}
	\item[$\mathbf{(A3)}$ Reidemeister 1 (de-looping)]
	\begin{align*}
		b_{(k), (k)}\circ i_{(k)}&= i_{(k)},\\
			&\hspace{-4.15em}\begin{tikzpicture}[auto,scale=.5,thick]
			\draw (0,0) to [out=180,in=0] ++(-2,1)
				to [out=180,in=180] ++(0,-1)
				to [out=0,in=180] ++ (2,1);
			\node (a) at (1,0.5) {$=$};
			\draw  (2.5,0) to [out=180,in=0] ++(-0.5,0)
				to [out=180,in=180] ++(0,1)
				to [out=0,in=180] ++(0.5,0);
		\end{tikzpicture}\quad.
	\end{align*}
		\item[(A4) Reidemeister 2 (double crossing)]
		\begin{align*}
		b_{(l), (k)}\circ b_{(k), (l)}&=1_{(k)}\otimes 1_{(l)},\\\vspace{1em}
			&\hspace{-6.45em}\begin{tikzpicture}[auto,scale=.5,thick]
				\draw[](0,0)to[out=0,in=180] ++(2,1)
					to[out=0,in=180] ++(2,-1);
				\draw[dashed] (0,1)to [out=0,in=180] ++(2,-1)
					to[out=0,in=180] ++(2,1);
				\node (a) at (5,0.5) {$=$};
				\draw[] (6,0) to [out=0,in=180] ++(2,0);
				\draw[dashed] (6,1) to [out=0,in=180] ++(2,0);
			\end{tikzpicture}\quad.
		\end{align*}
		\item[(A5) Reidemeister 3 (braiding 
		relation, a.k.a. Yang-Baxter equation)]
		\begin{align*}
		(1_{(l)}\otimes b_{(j), (k)})\circ(b_{(j), (l)}\otimes 
		1_{(k)})\circ(1_{(j)}\otimes b_{(k), (l)})&\\
		&\hspace{-7em}=
		(b_{(k), (l)}\otimes1_{(j)})\circ(1_{(k)}\otimes b_{(j), (l)})\circ(b_{(j), (k)}
		\otimes 1_{(l)}),
		\\\vspace{1em}
		&\hspace{-13.35em}
		\begin{tikzpicture}[auto,scale=.5,thick]
			\draw[] (0,0) to[out=0,in=180] ++(2,0)
				to[out=0,in=180] ++(2,1)
				to[out=0,in=180] ++(2,1);
			\draw[dashed] (0,1) to[out=0,in=180] ++(2,1)
				to[out=0,in=180] ++(2,0)
				to[out=0,in=180] ++(2,-1);
			\draw[dotted] (0,2) to[out=0,in=180] ++(2,-1)
				to[out=0,in=180] ++(2,-1)
				to[out=0,in=180] ++(2,0);
			\node (a) at (7,1) {$=$};
			\draw (8,0) to[out=0,in=180] ++(2,1)
				to[out=0,in=180] ++(2,1)
				to[out=0,in=180] ++(2,0);
			\draw[dashed] (8,1) to[out=0,in=180] ++(2,-1)
				to[out=0,in=180] ++(2,0)
				to[out=0,in=180] ++(2,1);
			\draw[dotted] (8,2) to[out=0,in=180] ++(2,0)
				to[out=0,in=180] ++(2,-1)
				to[out=0,in=180] ++(2,-1);
		\end{tikzpicture}\quad.
		\end{align*}
	\end{description}
\end{remark}

\subsection{Linear representations of the chromatic Brauer category}\label{Linear representations of the chromatic Brauer category}
For the purpose of constructing symmetric strict monoidal functors $Y\colon \mathbf{cBr}\rightarrow \mathbf{Vect}$, where $\mathbf{Vect}$ is equipped with the Schauenburg tensor product (see \Cref{The Schauenburg tensor product}), we use the notion of a \textit{duality structure} on a finite dimensional real vector space $V$.
Namely, a duality structure on $V$ is a pair $(i,e)$ whose components are a symmetric copairing $i\colon\mathbb{R}\rightarrow V\otimes V$ and a symmetric pairing $e\colon V\otimes V\rightarrow \mathbb{R}$, also called \textit{unit} and \textit{counit}, respectively, satisfying the \textit{zig-zag equation} $(e\otimes 1_V)\circ(1_V\otimes i)=1_V$.
In other words, $V$ is dualizable with dual $V$, and is hence self-dual.

Let $d$ be the dimension of $V$, and let $\{v_1,...,v_d\}$ be a basis of $V$.
Then, the set of all duality structures on $V$ is in 1-1 correspondence to the set of symmetric and invertible $(d\times d)$-matrices $\operatorname{Sym}(d,\mathbb{R})\cap\operatorname{GL}(d,\mathbb{R})$.
Indeed, let $e_{jk}=e(v_j\otimes v_k)$, and set $X=\operatorname{Mat}(e)=(e_{jk})_{j,k=1}^d$.
Then, $X$ is symmetric due to the symmetry of $e$, and -- by the zig-zag equation -- $X$ is invertible with inverse $X^{-1}= \operatorname{Mat}(i)=(i_{jk})_{j,k=1}^d$, where the coefficients $i_{jk}$ are given by $i(1)=\sum_{j,k=1}^d i_{jk}v_j\otimes v_k$.
Conversely, let $X\in \operatorname{Sym}(d,\mathbb{R})\cap\operatorname{GL}(d,\mathbb{R})$ be symmetric and invertible.
Then, the matrices $\operatorname{vec}(X^{-1})$ and $\operatorname{vec}(X)^{\operatorname{T}}$ define a symmetric copairing and pairing such that the zig-zag equation is satisfied, where $(-)^{\operatorname{T}}$ denotes the transposition of a matrix, and $\operatorname{vec}(-)$ denotes the vectorization of a matrix formed by stacking the columns of the matrix into a single column vector.
\\

For a linear representation $Y \colon \mathbf{cBr} \rightarrow \mathbf{Vect}$ of \textbf{cBr}, the pair $(Y(i_{(k)}),Y(e_{(k)}))$ forms a duality structure on $V_k$.
Hence, $Y((k))=V_k$ is for all $k\in\mathbb{N}$ a finite dimensional vector space (cf. Proposition 2.7 in \cite{ban3}).
Since $Y$ is symmetric, we also have $Y(b_{(k), (l)}) = b_{V_k,V_l}$, where $b_{V_k,V_l}\colon V_k \otimes V_l \rightarrow V_l \otimes V_k$ is the braiding isomorphism in $\mathbf{Vect}$, which is induced by $v_k\otimes v_l\mapsto v_l\otimes v_k$.
Conversely, the following result shows that we may construct symmetric strict monoidal functors $Y\colon\mathbf{cBr}\rightarrow\mathbf{Vect}$ by choosing duality structures.

\begin{theorem}[Linear representations of $\mathbf{cBr}$]\label{theorem linear representations of cBr}
Let $V_k$ be a finite dimensional real vector space and let the pair $(i^{(k)}, e^{(k)})$ be a duality structure on $V_k$ for all $k\in\mathbb{N}$.
Then there exists a unique symmetric strict monoidal functor $Y\colon (\mathbf{cBr},\otimes,I,b)\rightarrow (\mathbf{Vect}, \otimes, \mathbb{R},b)$, which satisfies $Y((k))=V_k$ and preserves duality, i.e. $Y(i_{(k)})=i^{(k)}$ and $Y(e_{(k)})=e^{(k)}$ for all $k\in\mathbb{N}$.
\end{theorem}

\begin{proof}
On object level, the values of $Y$ are uniquely and unambiguously determined by the requirement that $Y((k))= V_k$ for all $k\in\mathbb{N}$ because the objects of $\mathbf{cBr}$ are just the finite tensor products of objects of the form $(k)$, and $Y$ has to respect the monoidal structure.
(In particular, we have $Y(I)=\mathbb{R}$.)

By \Cref{generators and relations of cbr}, $\mathbf{cBr}$ is generated by the elementary morphisms $e_{(k)}$ $i_{(k)}$ and $b_{(k), (l)}$.
Thus, the values of $Y$on morphisms are uniquely determined by the requirements $Y(i_{(k)})=i^{(k)}$ and $Y(e_{(k)})=e^{(k)}$ for all $k \in \mathbb{N}$, and by the requirement $Y(b_{(k), (l)}) = b_{V_k,V_l}$, $k, l \in \mathbb{N}$, of being a symmetric monoidal functor.
To show that $Y$ is unambiguously defined on morphisms, we have to verify that the relations \textbf{(C1)} to \textbf{(C4)} of \Cref{definition cbr} are valid in $\mathbf{Vect}$ after applying $Y$ to every elementary morphism and every identity morphism in these relations.
In view of \Cref{generators and relations of cbr}, we can instead verify the relations \textbf{(A1)} to \textbf{(A5)} as follows.
The zig-zag relation \textbf{(A1)} is satisfied by the definition of the duality structure.
To verify relation \textbf{(A2)}, we fix a bases $\{v_1^{(k)},...,v_{d_k}^{(k)}\}$ and $\{v_1^{(l)},...,v_{d_l}^{(l)}\}$ of the finite dimensional vector spaces $V_k$ and $V_l$, respectively.
Then, the desired relation follows by noting that, for $\eta,\nu\in\{1,...,d_k\}$ and $\mu\in\{1,...,d_l\}$,
\begin{align*}
		(e^{(k)}\otimes 1_{V_l})((1_{V_k}\otimes b_{V_l,V_k})
		(v_\eta^{(k)}\otimes v_\mu^{(l)}\otimes v_\nu^{(k)}))
		&=(e^{(k)}\otimes 1_{V_l})(v_\eta^{(k)}\otimes v_\nu^{(k)}
		\otimes v_\mu^{(l)})\\
		&\hspace{-5em}=e^{(k)}(v_\eta^{(k)}\otimes v_\nu^{(k)})v_\mu^{(l)}\\
		&\hspace{-5em}=(1_{V_l}\otimes e^{(k)})(v_\mu^{(l)}
		\otimes v_\eta^{(k)}\otimes v_\nu^{(k)})\\
		&\hspace{-5em}=(1_{V_l}\otimes e^{(k)})((b_{V_k,V_l}\otimes 1_{V_k})
		(v_\eta^{(k)}\otimes v_\mu^{(l)}\otimes v_\nu^{(k)})).
\end{align*}
Relation \textbf{(A3)} is satisfied because $i$ is a symmetric copairing.
Moreover, the transposition $b$ clearly satisfies relation \textbf{(A4)}.
Finally, $b$ is a well-known solution of the \textit{Yang-Baxter equation} so that relation \textbf{(A5)} is satisfied as well.
\end{proof}

\section{Proof of \Cref{main theorem}}\label{proof of main theorem}

The proof of \Cref{main theorem} presented here is based on the proof of \Cref{main1 theorem} given in Chapter 5 of \cite{mue}.
It is divided into two parts as follows.
The first part culminates in \Cref{faithfulOnLoops}, which states that the given symmetric strict monoidal functor $Y\colon \mathbf{cBr}\rightarrow \mathbf{Vect}$ is faithful if and only if it is faithful on loops (a notion that will be defined below).
As it turns out, \Cref{faithfulOnLoops} is a consequence of \Cref{prop:loops}, which takes place solely in the chromatic Brauer category. 
Secondly, we prove \Cref{y is a faithful functor} which classifies all symmetric strict monoidal functors $Y$ which are faithful on loops.
\\

For $k \in \mathbb{N}$, we define the \emph{$k$-loop} $\lambda_{(k)}=e_{(k)} \circ i_{(k)}$, which is an endomorphism of the identity object $I \in \operatorname{Ob}(\mathbf{cBr})$.
We call a morphism $\varphi \colon ([m],c)\rightarrow([m'],c')$ in $\mathbf{cBr}$ \emph{loop-free} if there is no morphism $\varphi_0 \colon ([m],c) \rightarrow ([m'],c')$ such that $\varphi = \lambda_{(k)}\otimes \varphi_0$ for some $k \in \mathbb{N}$.
Then, we obtain the following normal form for morphisms in $\mathbf{cBr}$.

\begin{lemma}\label{lemma normal form}
Any morphism $\varphi \colon ([m],c)\rightarrow([m'],c')$ in $\mathbf{cBr}$ can be written for some suitable $(l_k)_{k\in\mathbb{N}} \in \bigoplus_{k=0}^\infty\mathbb{N}$ in the form
\begin{align*}
	\varphi
	= \left(\bigotimes_{k=0}^{\infty}\lambda_{(k)}^{\otimes l_{k}}\right)
	\otimes (\beta \circ \varphi_{0} \circ \alpha),
\end{align*}
where $\alpha \colon ([m],c) \rightarrow ([m],c_0)$ and
$\beta \colon ([m'],c'_0) \rightarrow ([m'],c')$ are isomorphisms in $\mathbf{cBr}$, 
and $\varphi_{0} \colon ([m],c_0) \rightarrow ([m'],c'_0)$ is a loop-free 
morphism in $\mathbf{cBr}$ that can be written for some suitable
$(p_k)_{k\in\mathbb{N}}, (q_k)_{k\in\mathbb{N}} \in
\bigoplus_{k=0}^\infty\mathbb N$ in the form
\begin{displaymath}
\varphi_{0} = \bigotimes_{k=0}^\infty 1_{([|c^{-1}(k)|-2p_k],k)} 
\otimes e_{(k)}^{\otimes p_k}\otimes i_{(k)}^{\otimes q_k}.
\end{displaymath}
(In particular, note that the maps $c_0\colon [m]\rightarrow \mathbb{N}$ and 
$c_0' \colon [m'] \rightarrow \mathbb{N}$ are monotone, and satisfy
$|c_0^{-1}(k)| - 2p_{k} = |c_0'^{-1}(k)| - 2q_{k}$ for all $k \in \mathbb{N}$.)
\end{lemma}

We introduce some operations $\varphi^{\operatorname{op}}$, $^o\varphi$ and $\varphi^o$ on morphisms $\varphi \colon ([m],c) \rightarrow ([m'],c')$ in $\mathbf{cBr}$ as follows.
Define $\varphi^{\operatorname{op}}$ by
\begin{align}\label{varphiop}
	\left(e_{([m'],c')}\otimes 1_{([m],c)}\right)
	\circ \left(1_{([m'],c')}\otimes \varphi\otimes 1_{([m],c')}\right)
	\circ \left(1_{([m'],c')}\otimes i_{([m],c)}\right).
\end{align}
Using the triangular equations \textbf{(C3)} one can equivalently define $\varphi ^{\operatorname{op}}$ by
\begin{align*}
\left(1_{([m],c)}\otimes e_{([m'],c')}\right) \circ \left(1_{([m],c)}\otimes \varphi \otimes 1_{([m'],c')} \right) \circ \left(i_{([m],c)}\otimes 1_{([m'],c')}\right).
\end{align*}
Intuitively, $\varphi^{\operatorname{op}}$ can be obtained from $\varphi$ by reflecting a diagram as considered in \Cref{remark cbr intuition} along a vertical axis.
Note, that for isomorphisms $\alpha \colon ([m],c) \rightarrow ([m],c')$ the identity $\alpha^{\operatorname{op}}=\alpha^{-1}$ holds.
Also note the validity of the equations $i_{(k)}^{\operatorname{op}}=e_{(k)}$ and $e_{(k)}^{\operatorname{op}}=i_{(k)}$.
Furthermore, we define $^o\varphi\colon([m],c)\otimes([m'],c')\rightarrow ([0],c_\emptyset)$ and $\varphi^o\colon ([0],c_\emptyset)\rightarrow ([m],c)\otimes([m'],c')$ by $^o\varphi:=e_{([m],c)}\circ(1_{([m],c)}\otimes \varphi^{\operatorname{op}})$ and $\varphi^o:=(1_{([m],c)}\otimes \varphi)\circ i_{([m],c)}$, respectively.
\\

\begin{proposition}\label{prop:loops}
Let $\varphi,\psi\colon ([m],c)\rightarrow ([m'],c')$ be loop-free morphisms in $\mathbf{cBr}$.
Then $\varphi= \psi$ if and only if
	\begin{align}\label{loops}
		^o \varphi\circ\psi^o
		=\bigotimes_{k=0}^\infty\lambda_{(k)}^{
		\otimes \frac12(|c^{-1}(k)|+|c'^{-1}(k)|)}.
	\end{align}
\end{proposition}

\begin{proof}
	If $\varphi=\psi$, then
	\begin{align*}
		^o\varphi\circ \varphi^o&=
		e_{([m],c)}\circ(1_{([m],c)}\otimes \varphi^{\operatorname{op}})
		\circ
		(1_{([m],c)}\otimes \varphi)\circ i_{([m],c)}\\
		&=e_{([m],c)}\circ (1_{([m],c)}\otimes 
		(\varphi^{\operatorname{op}}\circ \varphi))\circ i_{([m],c)}.
	\end{align*}
	Let us compute the expression $\varphi^{\operatorname{op}}\circ \varphi$
	seperately by using the normal form of $\varphi$ from \Cref{lemma normal form},
	\begin{align*}
		\varphi^{\operatorname{op}}\circ \varphi&=
		\alpha^{-1}\circ\left(\bigotimes_{k=0}^\infty1_{(
		[|c^{-1}(k)|-2p_k],\underline{k})}\otimes i_{(k)}^{\otimes p_k}\otimes 
		e_{(k)}^{\otimes q_k}\right)\circ \beta^{-1}\circ\\
		&\hspace{1.5em}\circ\beta\circ
		\left(\bigotimes_{k=0}^\infty
		1_{([|c^{-1}(k)|-2p_k],\underline{k})}\otimes
		e_{(k)}^{\otimes p_k}\otimes i_{(k)}^{\otimes q_k}\right)\circ \alpha\\
		&=\alpha^{-1}\circ
		\left(
		\bigotimes_{k=0}^\infty 1_{([|c^{-1}(k)|-2p_k],\underline{k})}
		\otimes (i_{(k)}\circ e_{(k)})^{\otimes p_k}\otimes
		(e_{(k)}\circ i_{(k)})^{\otimes q_k}
		\right)\circ \alpha\\
		&=
		\left[\alpha^{-1}\circ
		\left(
		\bigotimes_{k=0}^\infty 1_{([|c^{-1}(k)|-2p_k],\underline{k})}
		\otimes (i_{(k)}\circ e_{(k)})^{\otimes p_k}
		\right)\circ \alpha\right]
		\otimes \bigotimes_{k=0}^\infty\lambda_{(k)}^{\otimes q_k}.
	\end{align*}
	Now, for the isomorphism $\alpha$ we have the relations
	valid in $\mathbf{cBr}$ given by
	\begin{align*}
		e_{([m],c)}\circ (1_{([m],c)}\otimes \alpha^{-1})&=
		e_{([m],c_0)}\circ(\alpha\otimes 1_{([m],c_0)}),\,\text{and}\\
		(1_{([m],c)}\otimes \alpha)\circ i_{([m],c)}&=
		(\alpha^{-1}\otimes 1_{([m],c_0)})\circ i_{([m],c_0)}.
	\end{align*}
	(Indeed, if we interpret $\alpha$ as a bijection on $[m]$ and forgetting about $c$, we can write $\alpha$ as a product of adjacent transpositions
	$\alpha =T_1\circ\cdots\circ T_N$.
	Then $T_i$ can be shifted along $e_{([m],\star)}$ (resp. $i_{([m],\star)}$) from $T_i\otimes 1$ to $1\otimes T_i$, but in the reverse order.
	Note that during this procedure, the coloring $\star$ is
	changing after each step, after the last shift of $T_N$ it has become
	$c_0$.) This leads to
	\begin{align*}
		^o\varphi\circ \varphi^o\\
		&\hspace{-3.5em}=
		\left[
		e_{([m],c_0)}\circ\left(1_{([m],c_0)}\otimes\left(
		\bigotimes_{k=0}^\infty 1_{([|c^{-1}(k)|-2p_k],\underline{k})}
		\otimes (i_{(k)}\circ e_{(k)})^{\otimes p_k}\right)
		\right)
		\circ i_{([m],c_0)}\right]\otimes\\
		&\hspace{21em}
		\otimes\bigotimes_{k=0}^\infty\lambda_{(k)}^{\otimes q_k}\\
		&\hspace{-3.5em}=
		\left[
		e_{([2\sum p_k],\tilde c)}\circ\left(1_{([2\sum p_k],\tilde c)}\otimes
		\left(\bigotimes_{k=0}^\infty(i_{(k)}\circ e_{(k)})^{\otimes p_k}
		\right)\right)\circ i_{([2\sum p_k],\tilde c)}\right]\otimes\\
		&\hspace{21em}
		\otimes\bigotimes_{k=0}^\infty\lambda_{(k)}^{\otimes(q_k-2p_k+|
		c^{-1}(k)|)}
		\\
		&\hspace{-4em}\overset{\eqref{varphiop}}=
		\left[\left(\bigotimes_{k=0}^\infty i_{(k)}^{\otimes p_k}\right)^{\operatorname{op}}
		\circ
		\left(\bigotimes_{k=0}^\infty e_{(k)}^{\otimes p_k}\right)^{\operatorname{op}}
		\right]\otimes\bigotimes_{k=0}^\infty\lambda_{(k)}^{\otimes(q_k-2p_k+|
		c^{-1}(k)|)}\\
		&\hspace{-3.5em}=
		\bigotimes_{k=0}^\infty\lambda_{(k)}^{\otimes
		(q_k-p_k+|c^{-1}(k)|)}
		=\bigotimes_{k=0}^\infty\lambda_{(k)}^{\otimes
		\frac12(|c^{-1}(k)|+|c'^{-1}(k)|)},
	\end{align*}
	where $\tilde c\colon [2\sum_{k}p_k]\rightarrow \mathbb N$ is monotone
	and satisfies $|\tilde c^{-1}(k)|=2p_k$.\\

	Conversely, let $\varphi,\psi\colon ([m],c)\rightarrow ([m'],c')$ be two morphisms satisfying \eqref{loops}.
	Considering $^o\varphi$ and $\psi^o$ as being represented by embedded tangles in $[0,1] \times \mathbb{R}^3$ (compare \Cref{remark cbr intuition}) denoted by $W(^o\varphi)$ and $W(\psi^o)$, respectively, it suffices to show that for every component $C$ of $W(^a\varphi_0)$ there exists a component \reflectbox{$C$} of $W(\psi_0^a)$ such that $C$ and \reflectbox{$C$} have the same endpoints in $[m+m']$.
Let $P$ be a point in $[m+m']$ and let $P_\varphi$ and $P_\psi$ denote the other endpoint of the connected component $C$ and \reflectbox{$C$}, respectively, containing $P$.
Note first, that the number $\frac12(m+m')$ is the maximal number of loops $\{\lambda _{(k)} \, | \,  k \in \mathbb{N} \}$ which can be contained in $^a\varphi_0\circ\psi_0^a$, since $W(\psi_0^a)$ and $W(^a\varphi_0)$ each consist of $\frac12(m+m')$ distinguished connected components.
	This means that for every component $C$ of $W(^a\varphi_0)$ there
	is a component \reflectbox{$C$} of $W(\psi_0^a)$ such that
	$C$ and \reflectbox{$C$} close up to $S^1$. In other words,
	if $P\in C$ and $P\in\reflectbox{$C$}$, then $P_\varphi=P_\psi$.
\end{proof}

Let us now consider the given symmetric strict monoidal functor $Y\colon \mathbf{cBr}\rightarrow \mathbf{Vect}$.
Recall from the discussion in \Cref{Linear representations of the chromatic Brauer category} that $Y(([1],\underline{k}))=V_k$ is for all $k\in\mathbb{N}$ a finite dimensional vector space, whose dimension will be denoted by $d_k$.
By assumption, $d_{k} > 0$ for all $k$.

\begin{lemma}\label{lemma trace}
For all $k\in\mathbb{N}$ we have $Y(\lambda_{(k)})=d_k \; (>0)$.
\end{lemma}

\begin{proof}
By Proposition 2.9 in \cite{ban3}, the ``trace formula'' $e \circ i = d_{k}$ holds for any duality structure $(i, e)$ on $V_{k}$.
In particular, this applies to the duality structure $(Y(i_{(k)}), Y(e_{(k)}))$, and we have $Y(\lambda_{(k)}) = Y(e_{(k)}) \circ Y(i_{(k)})$.
\end{proof}

The symmetric strict monoidal functor $Y\colon \mathbf{cBr}\rightarrow \mathbf{Vect}$ is called \textit{faithful on loops} if for any two morphisms $\varphi, \psi\colon ([m],c)\rightarrow ([m'],c')$ in $\mathbf{cBr}$  the
condition $Y(\varphi)=Y(\psi)$ implies that there are a sequence $(l_k)_{k\in\mathbb N}\in \bigoplus_{k=0}^\infty\mathbb{N}$ and loop-free morphisms $\varphi_0$ and $\psi_0$ such that $\varphi=\bigotimes_k \lambda_{(k)}^{\otimes l_k}\otimes \varphi_0$ and $\psi=\bigotimes_k\lambda_{(k)}^{\otimes l_k}\otimes\psi_0$.

As an immediate consequence of \Cref{prop:loops} we obtain the following corollary.
\begin{corollary}\label{faithfulOnLoops}
	$Y\colon \mathbf{cBr}\rightarrow \mathbf{Vect}$ is faithful
	on loops if and only if $Y$ is faithful.
\end{corollary}

\begin{proof}
Let $\varphi_0,\psi_0\colon ([m],c)\rightarrow([m],c')$ be loop-free such that $Y(\varphi_0)=Y(\psi_0)$.
We need to show that $\varphi = \psi$.
Since $Y(\varphi) = Y(\psi)$, we have $Y(\varphi^o_0)=Y(\psi^o_0)$.
Therefore,
\begin{align*}
	Y(^o\varphi_0\circ \psi_0^o)
	&=Y(^o\varphi_0)\circ Y(\psi_0^o)
	=Y(^o\varphi_0)\circ 
	Y(\varphi_0^o)=Y(^o\varphi_0\circ\varphi_0^o)\\
	&\hspace{-1.3em}\overset{\text{Prop. } \ref{prop:loops}}
	=Y\left(\bigotimes_{k=0}^\infty \lambda_{(k)}^{\otimes\frac12
		(|c^{-1}(k)|+|c'^{-1}(k)|)}\right).
\end{align*}
Now, under the assumption that $Y$ is faithful on loops, \Cref{prop:loops} implies that $\varphi_0=\psi_0$.
Hence, $Y$ is also faithful on loop-free morphisms.
Now let $\varphi,\psi\colon ([m],c)\rightarrow([m'],c')$ be morphisms (possibly containing loops) such that $Y(\varphi)=Y(\psi)$, and rewrite them as $\varphi=\big(\bigotimes_k \lambda_{(k)}^{\otimes l_k}\big)\otimes \varphi_0$ and $\psi=\big(\bigotimes_k \lambda_{(k)}^{\otimes l_k}\big)\otimes \psi_0$, for a sequence $(l_k)_{k\in\mathbb N} \in \bigoplus_k\mathbb{N}$ and loop-free morphisms $\varphi_0, \psi_0\colon ([m],c)\rightarrow ([m'],c')$.
By \Cref{lemma trace}, we have in particular $Y(\lambda_{(k)})=d_k>0$ for all $k\in\mathbb{N}$.

Hence, we obtain
\begin{align*}
	Y(\varphi_0)=\frac1{\prod_k d_k^{l_k}} Y(\varphi)
	=\frac1{\prod_k d_k^{l_k}}Y(\psi)=Y(\psi_0),
\end{align*}
i.e. $\varphi_0=\psi_0$ and therefore $\varphi=\psi$.
\end{proof}

If $Y$ is faithful on loops, it is clear that the dimension $d_k$ of $V_k$ needs to satisfy for all $(l_k)_{k\in\mathbb{N}}\in\bigoplus_{k=0}^\infty\mathbb{Z}$ the implication \eqref{condition}, namely
\begin{displaymath}
	\prod_{k=0}^\infty d_k^{l_k}=1\qquad\Rightarrow \qquad
	l_k=0 \text{ for all $k\in\mathbb{N}$}.
\end{displaymath}

\begin{theorem}\label{y is a faithful functor}
Suppose that $d_k>0$ for all $k\in\mathbb{N}$, and that the implication 
\eqref{condition} holds for all sequences $(l_k)_{k\in\mathbb N}\in 
\bigoplus_{k=0}^\infty \mathbb Z$ . Then, the functor $Y\colon 
\mathbf{cBr}\rightarrow\mathbf{Vect}$ is faithful on loops.
\end{theorem}
\begin{proof}
Let $\varphi_0\colon ([m],c)\rightarrow([m],c')$ be a loop-free morphisms presented in its normal form $\varphi_0=\beta\circ\tilde\varphi_0\circ\alpha$ (see \Cref{lemma normal form}), with $\tilde\varphi_0=\bigotimes_{k}1_{([|c^{-1}(k)|-2p_k], \underline{k})}\otimes e_{(k)}^{\otimes p_k}\otimes i_{(k)}^{\otimes q_k}$.
We will compute the trace $\operatorname{Tr}(Y(\varphi_0^{\operatorname{op}}\circ \varphi_0))$ of the linear map $Y(\varphi_0^{\operatorname{op}}\circ \varphi_0)$.
Recall that the trace is invariant under cyclic permutation, and the trace of the tensor product of two matrices is the product of their traces.
We will also use the identity
$$
\operatorname{Tr} (i^{(k)}\circ e^{(k)}) = \operatorname{Tr} (e^{(k)}\circ i^{(k)})=e^{(k)}\circ i^{(k)} = Y(\lambda_{(k)}) \stackrel{\text{\Cref{lemma trace}}}{=} d_k.
$$
Thus, we obtain
\begin{align*}
	\operatorname{Tr}(Y(\varphi_0^{\operatorname{op}}\circ \varphi_0))
	&=\operatorname{Tr}\big(Y(\alpha^{-1})\circ
	Y(\tilde\varphi_0^{\operatorname{op}})\circ Y(\beta^{-1})\circ
	Y(\beta)\circ Y(\tilde\varphi_0)\circ Y(\alpha)\big)\\
	&=\operatorname{Tr}\big(Y(\tilde \varphi_0^{\operatorname{op}})\circ 
	Y(\tilde\varphi_0)\big)\\
	&=\operatorname{Tr}\left(Y\left(
	\bigotimes_{k=0}^\infty\lambda_{(k)}^{\otimes q_k}\otimes
		\left(
		\bigotimes_{k=0}^\infty 1_{([|c^{-1}(k)|-2p_k],\underline{k})}
		\otimes (i_{(k)}\circ e_{(k)})^{\otimes p_k}
		\right)
	\right)\right)\\
	&=\prod_{k=0}^\infty
	d_k^{q_k}\cdot \operatorname{Tr}(1_{V_k})^{|c^{-1}(k)|-2p_k}
	\cdot \operatorname{Tr}\big(i^{(k)}\circ e^{(k)}\big)^{p_k}\\
	&=\prod_{k=0}^\infty d_k^{|c^{-1}(k)|+q_k-p_k}
	=\prod_{k=0}^\infty d_k^{\frac12(|c^{-1}(k)|+|c'^{-1}(k)|)}.
\end{align*}
Note two things: The number $\operatorname{Tr}
(Y(\varphi_0^{\operatorname{op}}\circ \varphi_0))$ does not 
vanish in any case and it
only depends on the domain $([m],c)$ and codomain $([m],c')$ of the
morphism $\varphi_0$.

Now, let $\varphi,\psi\colon ([m],c)\rightarrow ([m'],c')$ such that $Y(\varphi)=Y(\psi)$.
Then there are sequences $(\mu_k)_{k\in\mathbb N}, (\nu_k)_{k\in\mathbb N} \in \bigoplus_{k=0}^\infty \mathbb N$ and loop-free morphisms $\varphi_0, \psi_0\colon ([m],c)\rightarrow ([m'],c')$ such that $\varphi=\bigotimes_{k=0}^{\infty} \lambda_{(k)}^{\otimes \mu_k}\otimes \varphi_0$ and $\psi=\bigotimes_{k=0}^{\infty} \lambda_{(k)}^{\otimes \nu_k}\otimes \psi_0$.
Then, $Y(\varphi)=Y(\psi)$ implies that $Y(\varphi^{\operatorname{op}})=Y(\psi^{\operatorname{op}})$ by \eqref{varphiop} and $Y$ being a strict monoidal functor.
Consequently,
\begin{align*}
	\prod_{k=0}^\infty d_k^{2\mu_k}\cdot \operatorname{Tr}
	(Y(\varphi_0\circ\varphi^{\operatorname{op}}_0))
	&=\operatorname{Tr}(Y(\varphi\circ\varphi^{\operatorname{op}}))\\
	&=\operatorname{Tr}(Y(\psi\circ\psi^{\operatorname{op}}))
	=\prod_{k=0}^\infty d_k^{2\nu_k}\cdot \operatorname{Tr}
	(Y(\psi_0\circ\psi_0^{\operatorname{op}})),
\end{align*}
which is equivalent to $\prod_kd_k^{2(\mu_k-\nu_k)}=1$. By implication 
\eqref{condition}, $\mu_k=\nu_k$ for all $k\in\mathbb{N}$. Hence, $Y$ is faithful 
on loops.
\end{proof}

This completes the proof of \Cref{main theorem}.

\begin{remark}[an alternative proof]
Based on Section 3.2 in \cite{wra}, we have an alternative proof of \Cref{y is a faithful functor} which does not make use of a trace argument, but exploits the Kronecker product of matrices instead.
\end{remark}

\section{Positive TFTs, fold maps, and exotic Kervaire spheres}\label{Linearization of the fold map positive TFTs}

Banagl \cite{ban2, ban3} has employed the Brauer category \textbf{Br} as well as singularity theory of fold maps to construct a high-dimensional positive TFT which is defined on smooth cobordisms.
As an application, he has shown that the state sum of his theory can distinguish exotic smooth structures on spheres from the standard smooth structure.
Banagl's construction is sketched in Section 10 of \cite{ban2} as an explicit example of the general framework of positive TFTs, and has been implemented in full detail in \cite{ban3}.
In the present section, we construct a refinement of Banagl's theory in which we replace the Brauer category \textbf{Br} by its chromatic enrichment \textbf{cBr}.
Our \Cref{theorem kervaire spheres} will illustrate the power of our state sum invariant by showing that the associated aggregate invariant can detect exotic Kervaire spheres in infinitely many dimensions.
On the other hand, we point out that the gluing axiom of positive TFTs ensures that our state sum invariant is computable by chopping cobordisms into pieces and computing their state sum invariants.

The present section is structured as follows.
In \Cref{general framework}, we outline the features of Banagl's general framework of positive TFTs, and explain the abstract process of quantization within the framework.
Next, \Cref{fields and action functional} provides the concrete definitions of fold fields and the \textbf{cBr}-valued action functional.
In particular, we point out the changes that arise from using the chromatic Brauer category instead of \textbf{Br}.
Then, quantization of our data is discussed in \Cref{quantization}, where we indicate carefully the necessary modifications in the algebraic process of so-called profinite idempotent completion.
Finally, in \Cref{aggregate invariant}, we define the aggregate invariant of homotopy spheres, and sketch our application to exotic Kervaire spheres.

\subsection{General framework}\label{general framework}
In \cite{ban2}, Banagl presents a new approach to the construction of certain TFTs in arbitrary dimension.
The basic idea is to modify Atiyah's original axioms \cite{ati} by formulating them over semirings instead of rings.
Compared to a ring, a semiring is not required to have additive inverses, i.e. ``negative'' elements.
As a result, Banagl coins the notion of \emph{positive} TFTs.
He shows that any system of so-called fields and action functionals gives rise to a positive TFT by means of a process he calls quantization in analogy with theoretical physics.
In order to avoid set theoretic difficulties that may arise in the definition of the Feynman path integral, Banagl employs the concept of complete semirings due to Eilenberg \cite{eil}.
The reason is that a complete semiring has a summation law that allows to sum families of elements indexed by arbitrary index sets.
Positive TFTs can motivate the construction of new invariants for smooth manifolds.
For example, Banagl defines the aggregate invariant of homotopy spheres (see Section 10 in \cite{ban3}). \\

In the following, we outline Banagl's construction of a $n$-dimensional positive TFT from given systems of fields and action functionals via the process of quantization (see Sections 4 to 6 in \cite{ban2}).

By definition, a system $\mathcal{F}$ of fields assigns to every closed $(n-1)$-manifold $M$ and to every $n$-cobordism $W$ sets of fields $\mathcal{F}(M)$ and $\mathcal{F}(W)$, respectively, such that certain properties hold.
Fields on a cobordism can be restricted to subcobordisms and to codimension $1$ submanifolds.
Moreover, fields behave in a desirable way with respect to the action of homeomorphisms and disjoint union.
Last but not least, fields are requested to glue under the gluing of cobordisms.
Next, we consider a system $\mathbb{T}$ of action functionals (or action exponentials) on fields $\mathcal{F}$ with values in a fixed strict monoidal category $\textbf{C}$.
In \Cref{fields and action functional}, we will specifically take $\textbf{C} = \textbf{cBr}$.
The notion of action functional is inspired by the exponential of the action that appears in the integrand of the Feynman path integral, and satisfies the following axioms.
To every $n$-cobordism $W$ one associates a map $\mathbb{T}_{W} \colon \mathcal{F}(W) \rightarrow \operatorname{Mor}(\textbf{C})$ in such a way that disjoint union of cobordisms is reflected by tensor product of morphisms in $\textbf{C}$, and gluing of cobordisms is reflected by composition of morphisms in $\textbf{C}$.
More precisely, one requires that $\mathbb{T}_{W}(f) = \mathbb{T}_{W'}(f|_{W'}) \otimes \mathbb{T}_{W''}(f|_{W''})$ for fields $f$ on the disjoint union $W = W' \sqcup W''$ of cobordisms $W'$ and $W''$, and $\mathbb{T}_{W}(f) = \mathbb{T}_{V}(f|_{V}) \circ \mathbb{T}_{U}(f|_{U})$ for fields $f$ on the gluing $W = U \cup_{N} V$ along $N$ of cobordisms $U$ from $M$ to $N$ and $V$ from $N$ to $P$.
Furthermore, the action functional is invariant under the action of homeomorphisms.
Eventually, let us describe the process of quantization.
For this purpose, we fix a system $\mathcal{F}$ of fields, a \textbf{C}-valued system $\mathbb{T}$ of action functionals on the fields, and a complete semiring $S$.
Following Section 4 in \cite{ban2}, one first constructs a complete additive monoid $Q = Q_{S}(\textbf{C})$ from the semiring $S$ and the strict monoidal category $\textbf{C}$.
The elements of $Q$ are just maps $\operatorname{Mor}(\textbf{C}) \rightarrow S$.
Then, one exploits the completeness of $S$ to define two different multiplications on $Q$.
As a result, one obtains a pair $(Q^{c}, Q^{m})$ of generally non-commutative complete semirings.
Multiplication in $Q^{c}$ is based on the composition of morphisms in $\textbf{C}$, whereas multiplication in $Q^{m}$ exploits the monoidal structure of $\textbf{C}$.
Next, as explained in Section 6 of \cite{ban2}, one assigns to every $n$-cobordism $W$ the composition $T_{W} \colon \mathcal{F}(W) \rightarrow Q$ of $\mathbb{T}_{W} \colon \mathcal{F}(W) \rightarrow \operatorname{Mor}(\textbf{C})$ with the map $\operatorname{Mor}(\textbf{C}) \rightarrow Q$ that assigns to every morphism $\gamma$ in $\textbf{C}$ its characteristic function $\chi_{\gamma}$.
Finally, the state sum $Z_{W} \colon \mathcal{F}(\partial W) \rightarrow Q$ is defined on a boundary condition $f \in \mathcal{F}(\partial W)$ as
\begin{displaymath}
Z_{W}(f) = \sum_{F \in \mathcal{F}(W, f)} T_{W}(F) \in Q,
\end{displaymath}
where the sum ranges over all fields $F$ on $W$ that extend $f$, i.e., $F|_{\partial W} = f$.
We note that $Z_{W}$ is well-defined due to the completeness of $Q$.
In analogy with the quantum Hilbert state from physics, the state module $Z(M)$ of a closed $n$-manifold $M$ consists of all maps (``states'') $\mathcal{F}(M) \rightarrow Q$ that satisfy a certain constraint equation.
It can be shown that $Z_{W}$ satisfies the constraint equation and is thus an element of the state module $Z(\partial W)$.
Furthermore, the state modules and state sums thus defined can be shown to satisfy Banagl's axioms of a positive TFT, including the essential gluing axiom.
For a topologically meaningful choice of fields and action functionals the state sum $Z_{W}$ provides an invariant of $n$-cobordisms $W$ that is interesting for further investigation.

\subsection{Fold fields and \textbf{cBr}-valued actions}\label{fields and action functional}

Fix an integer $n \geq 2$.
In this section, we specify the fields and actions that will determine our modification of the $n$-dimensional positive TFT constructed in \cite{ban3}.
All manifolds considered (with or without boundary) will be smooth, that is, differentiable of class $C^{\infty}$.

\subsubsection{Cobordisms}\label{cobordisms}

We recall from Section 7.1 of \cite{ban3} the terminology concerning manifolds and cobordisms.

From now on, we use the notation $M$, $N$, $P$ etc. for closed $(n-1)$-dimensional manifolds.
Fix an integer $D \geq 2n+1$.
We will always assume that any $M$ is smoothly embedded in $\mathbb{R}^{D}$, and that every connected component of $M$ is contained in a hyperplane of the form $\{k\} \times \mathbb{R}^{D-1}$ for some $k \in \left\{0, 1, 2, \dots\right\}$.

\begin{definition}\label{kobordismus}
A \textit{cobordism} from $M$ to $N$ is a compact $n$-dimensional smoothly embedded manifold $W \subset \left[0, 1\right] \times \mathbb{R}^{D}$ with the following properties: 
\begin{enumerate}[(1)]
\item the boundary of $W$ is $\partial W = M \sqcup N$, where $M \subset \mathbb{R}^{D} = \left\{0\right\} \times \mathbb{R}^{D}$ is the \emph{ingoing boundary} and $N \subset \mathbb{R}^{D} = \left\{1\right\} \times \mathbb{R}^{D}$ is the \emph{outgoing boundary},
\item the interior of $W$ satisfies $W \setminus \partial W \subset (0, 1) \times \mathbb{R}^{D}$,
\item there exists $0 < \varepsilon < \frac{1}{2}$ such that $W \cap \left[0, \varepsilon\right] \times \mathbb{R}^{D} = \left[0, \varepsilon\right] \times M$ and $W \cap \left[1 - \varepsilon, 1\right] \times \mathbb{R}^{D} = \left[1-\varepsilon, 1\right] \times N$ are product embeddings (any such $\varepsilon$ is referred to as a \textit{cylinder scale}), and
\item every connected component of $W$ is contained in a set of the form $\left[0, 1\right] \times \left\{k\right\} \times \mathbb{R}^{D-1}$ for some $k \in \left\{0, 1, 2, \dots\right\}$.
\end{enumerate}
\end{definition}

The advantage of working with embedded cobordisms $W \subset \left[0, 1\right] \times \mathbb{R}^{D}$ is that every such cobordism is naturally equipped with a \emph{time function} $\omega \colon W \rightarrow \left[0, 1\right]$ given by the restriction of the projection $\left[0, 1\right] \times \mathbb{R}^{D} \rightarrow \left[0, 1\right]$ to $W$.
For every regular value $t \in \left[0, 1\right]$ of the time function $\omega \colon W \rightarrow \left[0, 1\right]$ the preimage $\omega^{-1}(t)$ is a smoothly embedded codimension $1$ submanifold of $W$.

\subsubsection{System of fold fields}\label{System of fold fields}
Our theory will use exactly the same definition of fold fields on $n$-cobordisms that is used by Banagl in the original construction.
Thus, in this section we will outline the content of Section 7.2 of \cite{ban3}.
We also use the same sets of fields on closed $(n-1)$-manifolds although their definition relies on our modified action functional (see the end of \Cref{system of cbr-valued action functionals}).

The construction of fold fields on an $n$-dimensional cobordism $W$ is based on the notion of fold maps from $W$ into the plane $\mathbb{R}^{2} \cong \mathbb{C}$.
By definition, a \emph{fold map} of an $n$-manifold $X$ without boundary into the plane is a smooth map $F \colon X \rightarrow \mathbb{R}^{2}$ such that for every point $x \in X$ there exist coordinate charts centered at $x$ and $F(x)$ in which $F$ takes one of the following two normal forms:
\begin{align}\label{fold map normal form}
(t, x_{1}, \dots, x_{n-1}) \mapsto \begin{cases}
(t, x_{1}) &\text{(\emph{regular point} of $F$)}, \\
\left(t, -x_{1}^{2} - \dots - x_{i}^{2} + x_{i+1}^{2} + \dots + x_{n-1}^{2}\right) &\text{(\emph{fold point} of $F$)}.
\end{cases}
\end{align}
Let $S(F)$ denote the set of fold points of a fold map $F \colon X \rightarrow \mathbb{R}^{2}$.
It can be shown that $S(F) \subset X$ is a smoothly embedded $1$-dimensional submanifold that is closed as a subset, and that $F$ restricts to an immersion $S(F) \rightarrow \mathbb{R}^{2}$.
In analogy with the Morse index of non-degenerate critical points, there is the following notion of an \emph{(absolute) index} for fold points, which is a well-known application of the concept of intrinsic derivative (see Section VI.3 in \cite[p. 149ff]{gol}).

\begin{proposition}\label{proposition absolute index}
To any fold map $F \colon X \rightarrow \mathbb{R}^{2}$ one can associate a well-defined locally constant map
$$
\iota_{F} \colon S(F) \rightarrow \mathbb{N}, \qquad \iota_{F}(x) = \operatorname{min}\{i, n-1-i\},
$$
where $i \in \{0, \dots, n-1\}$ is the number of minus signs that appear in the local normal form of fold points (\ref{fold map normal form}).
\end{proposition}

Let $W$ be an $n$-dimensional cobordism from $M$ to $N$ in the sense of \Cref{kobordismus}.
A smooth map $F \colon W \rightarrow \mathbb{R}^{2}$ is called \emph{fold map} if $F$ has for some $\varepsilon > 0$ an extension to a fold map
$$
\widetilde{F} \colon ((-\varepsilon, 0] \times M) \cup_{\{0\} \times M} W \cup_{\{1\} \times N} ([1, 1+\varepsilon) \times N) \rightarrow \mathbb{R}^{2}.
$$
Given a fold map $F \colon W \rightarrow \mathbb{R}^{2}$, the intersection $S(\widetilde{F}) \cap W$ does not depend on the choice of the fold map extension $\widetilde{F}$, and will in the following be denoted by $S(F)$.
Furthermore, for an open subset $U \subset \partial W$ we write $S(F) \pitchfork U$ if $S(\tilde{F}) \pitchfork U$ for some (and hence, any) fold map extension $\tilde{F}$ of $F$.
If $S(F) \pitchfork \partial W$, then $S(F) \subset W$ is a $1$-dimensional smoothly embedded compact submanifold with boundary $\partial S(F) = S(F) \cap \partial W$.
In this case, we write $\iota_{F} \colon S(F) \rightarrow \mathbb{N}$ for the restriction of the absolute index $\iota_{\widetilde{F}} \colon S(\widetilde{F}) \rightarrow \mathbb{N}$ of \Cref{proposition absolute index} for some (and hence, any) fold map extension $\tilde{F}$ of $F$.
(The above notions can be shown to be independent of the choice of $\widetilde{F}$ by using the characterization of fold maps in terms of transversality given in Definition 4.1 in \cite[p. 87]{gol}.)

Let $\omega \colon W \rightarrow [0, 1]$ denote the time function associated to $W$ (see \Cref{cobordisms}).

\begin{definition}
Given a fold map $F \colon W \rightarrow \mathbb{C}$, we set
\begin{displaymath}
\pitchfork(F) = \left\{t \in \left[0, 1\right]; \; t \text{ is a regular value of } \omega, \text{ and } S(F) \pitchfork \omega^{-1}(t)\right\} \quad \subset \left[0, 1\right].
\end{displaymath}
\end{definition}

\begin{definition}\label{definition generic imaginary parts}
A fold map $F \colon W \rightarrow \mathbb{C}$ has \emph{generic imaginary parts} over $t \in \left[0, 1\right]$ if the restriction $\operatorname{Im} \circ F| \colon S(F) \cap \omega^{-1}(t) \rightarrow \mathbb{R}$ is injective. Let
\begin{displaymath}
\operatorname{GenIm}(F) = \left\{t \in \left[0, 1\right]; \; F \text{ has generic imaginary parts over } t\right\} \quad \subset \left[0, 1\right].
\end{displaymath}
\end{definition}

For $k \in \left\{0, 1, 2, \dots\right\}$ let $F(k)$ denote the restriction of a fold map $F \colon W \rightarrow \mathbb{C}$ to the part of $W$ that lies in $\left[0, 1\right] \times \left\{k\right\} \times \mathbb{R}^{D-1}$ (see \Cref{kobordismus}(4)):
\begin{displaymath}
F(k) = F| \colon W \cap (\left[0, 1\right] \times \left\{k\right\} \times \mathbb{R}^{D-1}) \rightarrow \mathbb{C}.
\end{displaymath}

Finally, fields on $W$ are fold maps $F \colon W \rightarrow \mathbb{C}$ with certain properties concerning the subsets $\pitchfork(F(k))$ and $\operatorname{GenIm}(F(k))$ of $\left[0, 1\right]$.

\begin{definition}[Definition 7.9 in \cite{ban3}]\label{fold field}
A \textit{fold field} on $W$ is a fold map $F \colon W \rightarrow \mathbb{C}$ such that for all $k \in \left\{0, 1, 2, \dots\right\}$ the following conditions hold:
\begin{enumerate}[$(1)$]
\item $0, 1 \in \; \pitchfork(F(k)) \cap \operatorname{GenIm}(F(k))$, and
\item $\operatorname{GenIm}(F(k))$ is residual in $[0, 1]$.
(That is, $\operatorname{GenIm}(F(k))$ contains the intersection of a countable family of dense open subsets of $[0, 1]$.)
\end{enumerate}
\end{definition}

Condition $(1)$ is exploited in the construction of the \textbf{Br}-valued action functional $\mathbb{S}$ in Section 7.3 of \cite{ban3} (as well as in our modified construction in \Cref{system of cbr-valued action functionals}).
Condition $(2)$ is crucial for the proof of the indispensable gluing theorem (see Section 7.7 in \cite{ban3}).

Let $\mathcal{F}(W)$ denote the set of all fold fields on $W$.
If $W = \emptyset$, then one puts $\mathcal{F}(W) = \{\ast\}$ (set with a single element).
Fields on closed $(n-1)$-dimensional manifolds will be introduced at the end of the following \Cref{system of cbr-valued action functionals}, which completes the definition of the system $\mathcal{F}$ of fields.

\subsubsection{System of $\mathbf{cBr}$-valued action functionals}\label{system of cbr-valued action functionals}

Banagl exploits singularity theory of fold maps into the plane to construct a system $\mathbb{S}$ of $\mathbf{Br}$-valued action functionals (see Section 7.3 in \cite{ban3}).
Namely, for every $n$-cobordism $W$ he constructs a function $\mathbb{S} \colon \mathcal{F}(W) \rightarrow \operatorname{Mor}(\mathbf{Br})$ assigning to every fold field on $W$ a morphism in \textbf{Br} that encodes the combinatorial information of the $1$-dimensional singular set of the fold map.
In the present section, we modify the original construction by replacing the Brauer category \textbf{Br} with its chromatic enrichment \textbf{cBr}.
The idea is to capture not only the singular patterns provided by the singular sets of fold fields, but also to remember the indices of fold lines by using labels from the set $\mathbb{N}$.
Hence, we construct a system $\overline{\mathbb{S}}$ of \textbf{cBr}-valued action functionals which is a lift of $\mathbb{S}$ under the forgetful map $\operatorname{Mor}(\mathbf{cBr}) \rightarrow \operatorname{Mor}(\mathbf{Br})$.

Let $W$ be an $n$-cobordism from $M$ to $N$.
We construct the function $\overline{\mathbb{S}} \colon \mathcal{F}(W) \rightarrow \operatorname{Mor}(\mathbf{cBr})$ as follows.
If $W$ is empty, then we set $\overline{\mathbb{S}}(\ast) = \operatorname{id}_{(\left[0\right], c_{\emptyset})}$.
Next suppose that $W$ is non-empty and entirely contained in a set of the form $\left[0, 1\right] \times \left\{k\right\} \times \mathbb{R}^{D-1}$, where $k \in \left\{0, 1, 2, \dots\right\}$.
Let $F \; (= F(k)) \in \mathcal{F}(W)$ be a field on $W$.
By condition (1) for fold fields (see \Cref{fold field}), we have $0, 1 \in \pitchfork(F)$, so that the intersections $S(F) \cap M$ and $S(F) \cap N$ are compact manifolds of dimension $0$.
Furthermore, since $F$ has generic imaginary parts over $0$ and $1$ (see \Cref{definition generic imaginary parts}), the composition $\operatorname{Im} \circ F \colon W \rightarrow \mathbb{R}$ restricts to injective maps on both $S(F) \cap M$ and $S(F) \cap N$.
Let $m$ and $m'$ denote the number of points in $S(F) \cap M$ and $S(F) \cap N$, respectively.
Then, we obtain orderings $S(F) \cap M = \left\{p_{1}, \dots, p_{m}\right\}$ and $S(F) \cap N = \left\{q_{1}, \dots, q_{m'}\right\}$ which are uniquely determined by requiring that $(\operatorname{Im} \circ F)(p_{i}) < (\operatorname{Im} \circ F)(p_{j})$ if and only if $i < j$, and $(\operatorname{Im} \circ F)(q_{i}) < (\operatorname{Im} \circ F)(q_{j})$ if and only if $i < j$.
The resulting bijections $S(F) \cap M \cong M[m]$, $p_{i} \mapsto i$, and $S(F) \cap N \cong M[m']$, $q_{i} \mapsto i$, are exactly the same as those described in the original construction of $\mathbb{S}$.
(Note that the notation $M[m] = [m]$ is used to indicate that $[m]$ is considered as a $0$-submanifold of $\mathbb{R}^{1}$.)
We define maps $c \colon [m] \rightarrow \mathbb{N}$ and $c' \colon [m'] \rightarrow \mathbb{N}$ by assigning to each point $x \in [m] = M[m] \cong S(F) \cap M$ and $x' \in [m'] = M[m'] \cong S(F) \cap N$ the index of the fold map $F$ at $x$ and $x'$, respectively (see \Cref{proposition absolute index}).
So far, we have constructed objects $(\left[m\right], c)$ and $(\left[m'\right], c')$ in \textbf{cBr}.
The desired morphism $\overline{\mathbb{S}}(F) \colon (\left[m\right], c) \rightarrow (\left[m'\right], c')$ in \textbf{cBr} is now represented in the sense of \Cref{remark cbr intuition} by the embedded tangle $S(F) \subset [0, 1] \times \mathbb{R}^{3}$ that is defined in exactly the same manner as described in the construction of $\mathbb{S}$, together with the labeling given by $\iota_{F}$ (see \Cref{proposition absolute index}).
That is, every component of $S(F)$ with non-empty boundary is embedded as a smooth arc that connects the corresponding points in $(\left\{0\right\} \times M\left[m\right]) \cup (\left\{1\right\} \times M\left[m'\right])$.
(For components of $S(F)$ without boundary one may choose an arbitrary embedding into $(0, 1) \times \mathbb{R}^{3}$.)
Finally, for an arbitrary non-empty cobordism $W$, we define $\overline{\mathbb{S}}(F) = \bigotimes_{k=0}^{\infty}\overline{\mathbb{S}}(F(k))$.
(Note that the tensor product is actually finite because $W$ is compact.)
This completes our construction of a system $\overline{\mathbb{S}}$ of \textbf{cBr}-valued action functionals which lifts $\mathbb{S}$ under the forgetful map $\operatorname{Mor}(\mathbf{cBr}) \rightarrow \operatorname{Mor}(\mathbf{Br})$.
Note that Lemma 7.12 in \cite{ban3} remains valid when replacing $\mathbb{S}$ with $\overline{\mathbb{S}}$ in the formulation.
That is, given a fold field $F$ on $W$ and some $t \in (0, 1)$ such that $t \in \pitchfork(F(k)) \cap \operatorname{GenIm}(F(k))$ for all $k \in \{0, 1, \dots\}$, $F$ restricts to fold fields $F_{\leq t}$ on $W \cap ([0, t] \times \mathbb{R}^{D})$ and $F_{\geq t}$ on $W \cap ([1-t, 1] \times \mathbb{R}^{D})$, and we have $\overline{\mathbb{S}}(F) = \overline{\mathbb{S}}(F_{\leq t}) \circ \overline{\mathbb{S}}(F_{\geq t})$ in \textbf{cBr}.

Finally, fields on a closed $(n-1)$-manifold $M$ are defined to be certain fold fields on the cylinder $[0, 1] \times M \subset [0, 1] \times \mathbb{R}^{D}$, i.e., the trivial cobordism from $M$ to $M$.
Namely, when $M$ is non-empty, we put
$$
\mathcal{F}(M) = \{f \in \mathcal{F}([0, 1] \times M); \; \overline{\mathbb{S}}(f) = 1 \in \operatorname{Mor}(\mathbf{cBr})\},
$$
where $1$ denotes some identity morphism in \textbf{cBr}.
Note that a fold field $f \in \mathcal{F}([0, 1] \times M)$ satisfies $\overline{\mathbb{S}}(f) = 1$ in \textbf{cBr} if and only if $\mathbb{S}(f) = 1$ in \textbf{Br}.
If $M = \emptyset$, then one puts $\mathcal{F}(M) = \{\ast\}$ (set with a single element).
Hence, the set of fields on closed $(n-1)$-manifold remains unchanged when replacing $\mathbb{S}$ with $\overline{\mathbb{S}}$.
In particular, Lemma 7.13 and Lemma 7.14 (additivity axiom) in \cite{ban3} remain valid when replacing $\mathbb{S}$ with our modified \textbf{cBr}-valued action functional $\overline{\mathbb{S}}$ in the formulation.

\subsubsection{Linearization}\label{linearization}
It can be advantageous to linearize a given category-valued system of action functionals because linear categories are easier to handle as pointed out in Section 8 of \cite{ban2}.
During the process of quantization in \Cref{quantization} below, we will employ a linearization $\overline{\mathbb{T}}$ of our system $\overline{\mathbb{S}}$ of \textbf{cBr}-valued action functionals from \Cref{system of cbr-valued action functionals}.
Such a linearization $\overline{\mathbb{T}}$ is a system of \textbf{Vect}-valued action functionals that is defined by means of a fixed linear representation $Y \colon \mathbf{cBr} \rightarrow \mathbf{Vect}$ by assigning to every $n$-cobordism $W$ the composition
\begin{displaymath}
\overline{\mathbb{T}}_{W} \colon \mathcal{F}(W) \stackrel{\overline{\mathbb{S}}}{\longrightarrow} \operatorname{Mor}(\textbf{cBr}) \stackrel{Y}{\longrightarrow} \operatorname{Mor}(\textbf{Vect}).
\end{displaymath}
Since the linear representation $Y \colon \mathbf{cBr} \rightarrow \mathbf{Vect}$ can be chosen to be faithful due to \Cref{main theorem}, there is no loss of informational content when working with the corresponding linearization of $\overline{\mathbb{S}}$.
As a result, our theory yields a refinement of Banagl's aggregate invariant of homotopy spheres which will be studied in \Cref{aggregate invariant}.

\subsection{Quantization}\label{quantization}
The purpose of the present section is to apply the process of quantization over the Boolean semiring $S = \mathbb{B}$ (see \Cref{example boolean semiring} below) to the system $\mathcal{F}$ of fold fields of \Cref{System of fold fields} and our system $\overline{\mathbb{T}}$ of \textbf{Vect}-valued action functionals on $\mathcal{F}$ (see \Cref{linearization}) along Banagl's general framework of positive TFTs outlined in \Cref{general framework}.
Hence, we obtain a refinement of Banagl's high-dimensional positive TFT defined on smooth cobordisms.
In \Cref{aggregate invariant}, we will study the associated aggregate invariant of homotopy spheres.

First of all, \Cref{on semirings} below provides the necessary algebraic background on semirings.
In \Cref{idempotent profinite completion}, we will modify the algebraic process of profinite idempotent completion (see Section 6 in \cite{ban3}) to represent loops of different colors in \textbf{cBr} by a countable family of loop parameters.
Finally, we proceed to define our positive TFT $\overline{Z}$ in \Cref{state module}, where we will specify the state modules $\overline{Z}(M)$ of closed $(n-1)$-manifolds $M$, and define the state sums of $n$-cobordisms $W$ as certain elements $\overline{Z}_{W} \in \overline{Z}(\partial W)$.

\subsubsection{Semirings and semimodules}\label{on semirings}
We collect here a number of concepts from the theory of semirings and semimodules that are relevant for the process of quantization.
For a detailed background, we refer to \cite{eil, kro, sak, gol85, kar}.
Our summary is based on the presentations in Section 2 of \cite{ban2} and Section 4 of \cite{ban3}.

Recall that a (commutative) monoid is a triple $M = (M, \ast, e)$, where $M$ is a set equipped with a (commutative) associative binary operation $\ast$ and a two-sided identity element $e \in M$, that is, $e \ast m = m \ast e = m$ for all $m \in M$.
A semiring is a tuple $S = (S, +, \cdot, 0, 1)$, where $S$ is a set together with two binary operations $+$ and $\cdot$ and two elements $0, 1 \in S$ such that $(S, +, 0)$ is a commutative monoid, $(S, \cdot, 1)$ is a monoid, the multiplication $\cdot$ distributes over the addition from either side, and $0$ is absorbing, i.e. $0 \cdot s = 0 = s \cdot 0$ for every $s \in S$.
The semiring S is called commutative if the monoid $(S,\cdot,1)$ is commutative.
A morphism of semirings sends $0$ to $0$, $1$ to $1$ and respects addition and multiplication.
Fix a semiring $S$.
A (left) $S$-semimodule is a commutative monoid $M = (M,+,0_{M})$ together with a scalar multiplication $S \times M \rightarrow M$, $(s,m) \rightarrow sm$, such that for all $r, s \in S$, $m, n \in M$, we have $(rs)m = r(sm)$, $r(m+n) = rm+rn$, $(r+s)m = rm+sm$, $1m = m$, and $r0_{M} = 0_{M} = 0m$.
Given a morphism $\varphi \colon S \rightarrow T$ of semirings, it is clear that $T$ becomes a $S$-semimodule via $st = \varphi(s)t$.

A monoid $(M, \ast, e)$ is called idempotent if $m \ast m = m$ for all elements $m \in M$.
The semiring $(S, +, \cdot, 0, 1)$ is idempotent if $(S,+,0)$ is an idempotent monoid.
A semimodule is called idempotent if its underlying additive monoid is idempotent.

Next, we discuss the important notion of Eilenberg-completeness \cite[p. 125]{eil} for semirings and semimodules (see also \cite{kar, kro}).
A complete monoid is a commutative monoid $(M,+,0)$ together with an assignment $\Sigma$, called a summation law, which assigns to every family $(m_{i})_{i \in I}$, indexed by an arbitrary set $I$, an element $\sum_{i \in I} m_{i}$ of $M$ (called the sum of the $m_{i}$), such that
$$
\sum_{i \in \emptyset} m_{i} = 0, \quad \sum_{i \in \{1\}} m_{i} = m_{1}, \quad \sum_{i \in \{1, 2\}} m_{i} = m_{1} + m_{2},
$$
and for every partition $I = \bigcup_{j \in J} I_{j}$, we have
$$
\sum_{j \in J}\left(\sum_{i \in I_{j}} m_{i} \right) = \sum_{i \in I} m_{i}.
$$
A complete semiring is a semiring $S$ for which $(S,+,0, \Sigma)$ is a complete monoid, and infinite distributivity holds, that is,
$$
\sum_{i \in I} s s_{i} = s \left(\sum_{i \in I} s_{i}\right), \qquad \sum_{i \in I} s_{i} s = \left(\sum_{i \in I} s_{i}\right) s.
$$
A semimodule $M$ over a commutative semiring $S$ is called complete if its underlying additive monoid is equipped with a summation law that makes it complete as a commutative monoid, and infinite distributivity
$$
\sum_{i \in I} sm_{i} = s \left(\sum_{i \in I} m_{i}\right)
$$
holds for every $s \in S$ and every family $(m_{i})_{i \in I}$ in $M$.
If $\varphi \colon S \rightarrow T$ is a morphism of semirings and $T$ is complete as a semiring, then $T$ can be easily seen to be complete as an $S$-semimodule.

We will also need the following notion of continuity for idempotent complete semirings (cf. \cite{kro, sak, gol85, kar}).
Here, we only state the definition, and refer to the summary preceding Proposition 4.2 in \cite{ban3} for more details.
Observe that any idempotent monoid $(M, \ast, e)$ admits a natural partial order $\leq$ given by $m \leq m'$ if and only if $m + m' = m'$.
An idempotent complete monoid $(M,+,0, \Sigma)$ is continuous if for all families $(m_{i})_{i \in I}$, $m_{i} \in M$, and for all $c \in M$, $\sum_{i \in F} m_{i} \leq c$ for all finite $F \subset I$ implies $\sum_{i \in I} m_{i} \leq c$.
An idempotent complete semiring (semimodule) is called continuous if its underlying additive monoid is continuous.

It is useful to note that the product $\prod_{i \in I}M_{i}$ of a family $\{M_{i}\}_{i \in I}$ of continuous idempotent complete monoids is a continuous idempotent complete monoid.

\begin{example}\label{example boolean semiring}
The minimal example of a semiring that is not a ring is given by the Boolean semiring $\mathbb{B}$.
This is the set $\mathbb{B} = \left\{0, 1\right\}$ equipped with addition defined by $1 + 1 = 1$ and multiplication given by $0 \cdot 0 = 0$ (where $0$ and $1$ serve as identity elements for addition and multiplication, respectively).
Distributivity holds, but in $\mathbb{B}$ there exists no additive inverse for $1$.
We leave it to the reader to check that the commutative semiring $\mathbb{B}$ is idempotent, complete, and continuous.
\end{example}

\subsubsection{Profinite idempotent completion}\label{idempotent profinite completion}
As mentioned before, our input data for the process of quantization are the Boolean semiring $S = \mathbb{B}$ (see \Cref{example boolean semiring}), the system $\mathcal{F}$ of fold fields (see \Cref{System of fold fields}), and our system $\overline{\mathbb{T}}$ of \textbf{Vect}-valued action functionals on $\mathcal{F}$ (see \Cref{linearization}).
Recall from \Cref{general framework} that the first step of quantization consists in the construction of a pair $(Q^{c}, Q^{m})$ of generally non-commutative complete semirings.
In principle, we could follow the general construction provided in Section 4 of \cite{ban2}.
However, we would like to take care of the additional structure given by the natural action of the polynomial semiring of colored loops on the morphism sets of the chromatic Brauer category.
As a consequence, our state sum invariant will take values in power series in several variables with matrix coefficients.
Banagl has implemented this algebraic process of profinite idempotent completion in Section 6 of \cite{ban3} for the natural action of the polynomial semiring of loops on the morphism sets of the Brauer category \textbf{Br}.
Since we work instead with the chromatic Brauer category \textbf{cBr}, we will in the following discuss the modifications of Banagl's construction in detail.

First, we give a brief outline of Banagl's profinite idempotent completion.
Note that the sets $\operatorname{Hom}_{\mathbf{Br}}([m], [m'])$ have the special property that they are naturally equipped with the action $\tau^{i}\varphi = \varphi \otimes \lambda^{\otimes i}$ of the (multiplicatively written) monoid $\mathbb{N} = \{\tau^{i}; \; i \in \mathbb{N}\}$.
Fix a linear representation $U \colon \mathbf{Br} \rightarrow \mathbf{Vect}$ and write $V = U([1])$ and $\widehat{\lambda} = U(\lambda) \in \mathbb{R}$.
Then, the subset
$$
H_{m,n} = U(\operatorname{Hom}_{\mathbf{Br}}([m], [m']))
$$
of the real vector space $\operatorname{Hom}_{\mathbf{Vect}}(V^{\otimes m}, V^{\otimes m'})$ inherits an action of the monoid $\mathbb{N}$ via $\tau^{i}f = \widehat{\lambda}^{i}f$.
Given a set $A$, let $FM(A)$ denote the free commutative monoid generated by $A$.
In particular, $FM(H_{m, n})$ has the structure of a $\mathbb{N}[\tau]$-semimodule by Lemma 4.1 in \cite{ban3}.
A $\mathbb{N}[\tau]$-semimodule is given by the algebraic tensor product
$$
Q(H_{m, n}) = FM(H_{m, n}) \otimes_{\mathbb{N}[\tau]} \mathbb{B}[[q]],
$$
where $\mathbb{B}[[q]]$ is the semiring of formal power series associated to the Boolean semiring $\mathbb{B}$.
In Lemma 6.7 in \cite{ban3}, it is shown that $Q(H_{m, n})$ is isomorphic as a $\mathbb{N}[\tau]$-semimodule to a finite sum of copies of $\mathbb{B}[[q]]$, so that its elements consist of a number of power series in the \emph{loop parameter} $q$.
(This fact can be derived more abstractly by using the concept of minimal shells of projectively finite subsets of a real vector space, see Definition 6.1 in \cite{ban3}.)
Finally, the \emph{profinite idempotent completion} of the set $U(\operatorname{Mor}(\mathbf{Br}))$ is the $\mathbb{N}[\tau]$-semimodule
$$
Q = Q(U) = \prod_{m, m' \in \mathbb{N}}Q(H_{m, n}).
$$
Provided that the underlying functor $U$ is chosen to be faithful on loops, the additive monoid $(Q, +, 0)$ can be promoted to idempotent complete semirings $Q^{c}$ (the \emph{composition semiring}, see Proposition 6.12 in \cite{ban3}) and $Q^{m}$ (the \emph{monoidal semiring}, see Proposition 6.14 in \cite{ban3}).
It can be shown that the semirings $Q^{c}$ and $Q^{m}$ are both continuous (see Proposition 6.15 in \cite{ban3}).
Continuity is exploited in \cite{ban3} to check several axioms of positive TFTs, namely the behavior of state sums under disjoint union (Proposition 7.22), and the gluing axiom (Theorem 7.26). \\

We proceed to explain how the process of profinite idempotent completion applies to our setting.
When using the chromatic Brauer category instead of \textbf{Br}, we need to replace the semimodule $\mathbb{B}[[q]]$ over $\mathbb{N}[\tau]$ by the semimodule $\mathbb{B}[[\mathfrak{q}]] = \mathbb{B}[[q_{0}, q_{1}, \dots]]$ over $\mathbb{N}[\mathfrak{t}] = \mathbb{N}[\tau_{0}, \tau_{1}, \dots]$.
Here, the different parameters represent loops of different labels in \textbf{cBr}.
Let us introduce the semirings $\mathbb{N}[\tau]$ and $\mathbb{B}[[q]]$ (compare Section 4 in \cite{ban3}).
Set $\underline{\mathbb{N}} = \bigoplus_{i = 0}^{\infty}\mathbb{N}$, which is a commutative monoid with respect to component wise addition and identity element $0 = (0, 0, \dots)$.
In general, a \emph{(formal) power series} in a countable number of indeterminates $\mathfrak{q} = q_{0}, q_{1}, \dots$ and with coefficients in the Boolean semiring $\mathbb{B}$ is a function $a \colon \underline{\mathbb{N}} \rightarrow \mathbb{B}$, written as a formal sum $\sum_{\nu \in \underline{\mathbb{N}}} a(\nu) \mathfrak{q}^{\nu}$, where $\mathfrak{q}^{\nu}$ denotes the (finite) product $\prod_{s=0}^{\infty}q_{s}^{\nu_{s}}$.
The element $a(\nu)$ is referred to as the coefficient of $\mathfrak{q}^{\nu}$.
Let $\mathbb{B}[[\mathfrak{q}]]$ be the set of all power series over $\mathbb{B}$ having a countable number of indeterminates $\mathfrak{q} = q_{0}, q_{1}, \dots$.
We write $0$ for the power series $a$ with $a(\nu) = 0$ for all $\nu$, and $1$ for the power series $a$ with $a(0)=1$ and $a(\nu) = 0$ for all $\nu \neq 0$.
Define an addition on power series by $a+b=c$, where $c(\nu) = a(\nu) + b(\nu)$ for all $\nu$.
Define a multiplication on power series by the Cauchy product, that is, $a \cdot b = c$ where $c(\nu) = \sum_{\mu+\kappa = \nu}a(\mu)b(\kappa)$.
Then, $(\mathbb{B}[[\mathfrak{q}]], +, \cdot, 0, 1)$ is a commutative idempotent semiring, the semiring of power series over $\mathbb{B}$ in a countable number of indeterminates.
In a similar way, using finite sums rather than power series, one defines the polynomial semiring $\mathbb{N}[\mathfrak{t}]$ in a countable number of indeterminates $\mathfrak{t} = \tau_{0}, \tau_{1}, \dots$.
Note that $\mathbb{B}[[\mathfrak{q}]]$ is a $\mathbb{N}[\mathfrak{t}]$-semimodule via the semiring morphism $\mathbb{N}[\mathfrak{t}] \rightarrow \mathbb{B}[[\mathfrak{q}]]$ that extends the unique semiring morphism $\mathbb{N} \rightarrow \mathbb{B}$ by $\tau_{k} \mapsto q_{k}$, $k \in \mathbb{N}$.
It can be shown that $\mathbb{B}[[\mathfrak{q}]]$ is a complete semiring, and is hence complete as a $\mathbb{N}[\mathfrak{t}]$-semimodule.
Furthermore, the idempotent complete semiring $\mathbb{B}[[\mathfrak{q}]]$ can be shown to be continuous.
(Both claims follow easily from the material of \Cref{on semirings} by using that the monoidal structure of $\mathbb{B}[[\mathfrak{q}]]$ is isomorphic to the product $\prod_{\nu \in \underline{\mathbb{N}}}\mathbb{B}$.)

Returning to the category \textbf{cBr}, we observe that the $k$-loops $\lambda_{(k)}$, $k \in \mathbb{N}$, induce an action of the (multiplicatively written) commutative monoid $\underline{\mathbb{N}} = \{\mathfrak{t}^{\nu}; \; \nu \in \underline{\mathbb{N}}\}$ on the morphism sets
$$
\operatorname{Hom}_{\mathbf{cBr}}(([m],c), ([m'],c'))
$$
via $(\mathfrak{t}^{\nu}, \varphi) \mapsto \mathfrak{t}^{\nu} \varphi = \left(\bigotimes_{k=0}^{\infty} \lambda_{(k)}^{\otimes \nu_{k}}\right) \otimes \varphi$.
Fix a linear representation $Y \colon \mathbf{cBr} \rightarrow \mathbf{Vect}$.
Let us write $V_{k} = Y(([1], \underline{k}))$ and $\widehat{\lambda}_{(k)} = Y(\lambda_{(k)}) \in \mathbb{R}$.
Then, the subset
$$
H_{([m],c), ([m'],c')} = Y(\operatorname{Hom}_{\mathbf{cBr}}(([m],c), ([m'],c')))
$$
of the real vector space
$$\operatorname{Hom}_{\mathbf{Vect}}(Y(([m], c)), Y(([m'], c'))) = \operatorname{Hom}_{\mathbf{Vect}}(V_{c(1)} \otimes \dots \otimes V_{c(m)}, V_{c'(1)} \otimes \dots \otimes V_{c'(m')})$$
inherits an action of the monoid $\underline{\mathbb{N}}$ via $\mathfrak{t}^{\nu}f = \left(\prod_{k=0}^{\infty}\widehat{\lambda}_{(k)}^{\nu_{k}}\right) \cdot f$.
In analogy with Lemma 4.1 in \cite{ban3}, it follows that $FM(H_{([m],c), ([m'],c')})$, the free commutative monoid generated by $H_{([m],c), ([m'],c')}$, has the structure of a $\mathbb{N}[\mathfrak{t}]$-semimodule via
$$\sum_{\nu\in\underline{\mathbb{N}}} m_{\nu}\mathfrak{t}^{\nu} \cdot \sum_{j} \alpha_{j} f_{j} = \sum_{\nu, j} (m_{\nu} \alpha_{j}) (\mathfrak{t}^{\nu} f_{j}), \quad m_{\nu}, \alpha_{j} \in \mathbb{N}, f_{j} \in H_{([m],c), ([m'],c')}.$$
Using the algebraic tensor product of semimodules over the commutative semiring $\mathbb{N}[\mathfrak{t}]$ (see \cite{kat, kat2}, and compare Section 4 in \cite{ban3}) we can now define a $\mathbb{N}[\mathfrak{t}]$-semimodule by
$$
\overline{Q}(H_{([m],c), ([m'],c')}) = FM(H_{([m],c), ([m'],c')}) \otimes_{\mathbb{N}[\mathfrak{t}]} \mathbb{B}[[\mathfrak{q}]].
$$
Let $OP_{([m],c), ([m'],c')}$ denote the (finite) set of loop-free morphisms $([m],c) \rightarrow ([m'],c')$ in \textbf{cBr}.
From now on, we suppose that the linear representation $Y \colon \mathbf{cBr} \rightarrow \mathbf{Vect}$ is chosen to be faithful (see \Cref{linearization}).
Then, it can be shown that $\overline{Q}(H_{([m],c), ([m'],c')})$ is isomorphic in the category of $\mathbb{N}[\mathfrak{t}]$-semimodules to the product of copies of $\mathbb{B}[[\mathfrak{q}]]$ indexed by the elements of $OP_{([m],c), ([m'],c')}$.
In fact, in analogy to Lemma 6.6 in \cite{ban3}, we have the following
\begin{lemma}
Every element in $FM(H_{([m],c), ([m'],c')})$ can be written as
\begin{align}\label{standard presentation}
\sum_{\sigma=1}^{r} p_{\sigma}(\mathfrak{t})Y(\varphi_{\sigma})
\end{align}
for suitable polynomials $p_{\sigma}(\mathfrak{t}) \in \mathbb{N}[\mathfrak{t}]$, where $\varphi_{1}, \dots, \varphi_{r}$ is the list of elements of $OP_{([m],c), ([m'],c')}$.
If $Y$ is faithful, then the presentation in the form (\ref{standard presentation}) is unique, and thus $FM(H_{([m],c), ([m'],c')})$ is a free $\mathbb{N}[\mathfrak{t}]$-semimodule of rank $r$.
\end{lemma}

\begin{proof}
The first claim follows from the fact that every element of the morphism set $\operatorname{Hom}_{\mathbf{cBr}}(([m],c), ([m'],c'))$ can be written as $\left(\bigotimes_{k} \lambda_{(k)}^{\otimes l_{k}}\right) \otimes \varphi_{\sigma}$ for some $\sigma \in \{1, \dots, r\}$ and a sequence $(l_{k})_{k \in \mathbb{N}} \in \bigoplus_{k} \mathbb{N}$ (see \Cref{lemma normal form}).
To show uniqueness of presentations in the form (\ref{standard presentation}), let us suppose that $\sum_{\sigma=1}^{r} p_{\sigma}(\mathfrak{t})Y(\varphi_{\sigma}) = 0$ for some polynomials $p_{\sigma}(\mathfrak{t}) = \sum_{\nu \in \underline{\mathbb{N}}}a_{\sigma}^{(\nu)}\mathfrak{t}^{\nu} \in \mathbb{N}[\mathfrak{t}]$.
Then, using the definition of $FM(H_{([m],c), ([m'],c')})$ and the faithfulness of $Y$, it follows from
$$
0 = \sum_{\sigma=1}^{r} \left(\sum_{\nu \in \underline{\mathbb{N}}}a_{\sigma}^{(\nu)}\mathfrak{t}^{\nu}\right)Y(\varphi_{\sigma}) = \sum_{\sigma=1}^{r} \sum_{\nu \in \underline{\mathbb{N}}}a_{\sigma}^{(\nu)}Y\left(\left(\bigotimes_{k=0}^{\infty} \lambda_{(k)}^{\otimes \nu_{k}}\right) \otimes\varphi_{\sigma}\right)
$$
that $a_{\sigma}^{(\nu)} = 0$ for all $\sigma \in \{1, \dots, r\}$, $\nu \in \underline{\mathbb{N}}$, and the claim follows.
\end{proof}

Finally, the \emph{profinite idempotent completion} of the set $Y(\operatorname{Mor}(\mathbf{cBr}))$ is the $\mathbb{N}[\tau]$-semimodule
$$
\overline{Q} = \overline{Q}(Y) = \prod_{([m],c), ([m'],c')}\overline{Q}(H_{([m],c), ([m'],c')}).
$$

Being the product of copies of $\mathbb{B}[[\mathfrak{q}]]$, each $\overline{Q}(H_{([m],c), ([m'],c')})$ is a continuous idempotent complete monoid.
We conclude that the additive monoid $(\overline{Q}, +, 0)$ is continuous, idempotent and complete as well.
Hence, $(\overline{Q}, +, 0)$ can be advanced to continuous idempotent complete semirings $\overline{Q}^{c}$ and $\overline{Q}^{m}$ in analogy with the construction in Section 6 in \cite{ban3}.

\subsubsection{State modules and state sums}\label{state module}
Let $\overline{Q}$ denote the profinite idempotent completion of the set $Y(\operatorname{Mor}(\mathbf{cBr}))$ associated to a fixed faithful linear representation $Y \colon \mathbf{cBr} \rightarrow \mathbf{Vect}$ as constructed in \Cref{idempotent profinite completion}.
We proceed to define our smooth positive TFT $\overline{Z}$.
The state module of a closed $(n-1)$-manifold is defined to be $\overline{Z}(M) = \{z \colon \mathcal{F}(M) \rightarrow \overline{Q}\}$.
By Proposition 3.1 in \cite{ban2}, $\overline{Z}(M)$ inherits the structure of a two-sided $\overline{Q}^{c}$-semialgebra and a two-sided $\overline{Q}^{c}$-semialgebra, and $\overline{Z}(M)$ is complete.
Then, it follows from the corresponding properties of $\overline{Q}$ that $\overline{Z}(M)$ is idempotent and continuous.
The construction of a contraction product
$$
\langle \cdot, \cdot \rangle \colon (\overline{Z}(M) \widehat{\otimes} \overline{Z}(N)) \times (\overline{Z}(N) \widehat{\otimes} \overline{Z}(P)) \rightarrow \overline{Z}(M) \widehat{\otimes} \overline{Z}(P)
$$
is analogous to the discussion in Section 7.4 of \cite{ban3}.
Here, $\widehat{\otimes}$ denotes the complete tensor product of complete idempotent continuous semimodules (see Section 5 in \cite{ban3}) rather than the algebraic tensor product $\otimes$ of function semimodules discussed in \cite{ban1}.

Let $W^{n}$ be a cobordism from $M$ to $N$ in the sense of \Cref{kobordismus}.
The state sum will be defined as an element $\overline{Z}_{W} \in \overline{Z}(M) \widehat{\otimes} \overline{Z}(N)$.
Fix a cylinder scale $\varepsilon_{W}$ for $W$.
Given a boundary condition $(f_{M}, f_{N}) \in \mathcal{F}(M) \times \mathcal{F}(N)$, we define
\begin{align*}
\mathcal{F}(W; f_{M}, f_{N}) &= \{F \in \mathcal{F}(W) | \; \exists \varepsilon(k), \varepsilon'(k) \in (0, \varepsilon_{W}) \colon \\
&F|_{[0,\varepsilon(k)] \approx M(k)} \approx f_{M}(k), F|_{[1-\varepsilon'(k),1]\times N(k)} \approx f_{N}(k), \forall k\},
\end{align*}
where the equivalence relation $\approx$ for fold fields on closed $(n-1)$-manifolds $X$ is defined as follows.
Two smooth maps $f \colon [a,b] \times X \rightarrow \mathbb{C}$ and $f' \colon [a',b'] \times X \rightarrow \mathbb{C}$ are equivalent, $f \approx f'$, if there exists a diffeomorphism $\xi \colon [a,b] \rightarrow [a',b']$ with $\xi(a) = a'$ such that $f(t,x) = f'(\xi(t),x)$ for all $(t,x) \in [a,b] \times X$ (see Definition 7.18 in \cite{ban3}).
On $(f_{M}, f_{N})$ the state sum $\overline{Z}_{W}$ is then defined as
$$
\overline{Z}_{W}(f_{M}, f_{N}) = \sum_{F \in \mathcal{F}(f_{M}, f_{N})} \overline{\mathbb{T}}_{W}(F),
$$
which is a well-defined element of the complete semiring $\overline{Q}$.
Note that, when $\overline{\mathbb{S}}(F) \colon ([m], c) \rightarrow ([m'],c')$ in \textbf{cBr}, the element $\overline{\mathbb{T}}_{W}(F)$ is supposed to be identified with the element $\overline{\mathbb{T}}_{W}(F) \otimes 1 \in \overline{Q}(H_{([m],c), ([m'],c')}) \subset \overline{Q}$.

In close analogy with the further steps in \cite{ban3}, one can prove that our assignment $\overline{Z}$ is in fact a positive topological field theory.
Namely, following Section 7.6 in \cite{ban3}, one checks the correct behavior of our state sum under disjoint union.
Moreover, following Section 7.7 in \cite{ban3}, one proves the essential gluing formula $\overline{Z}_{W} = \langle \overline{Z}_{W'}, \overline{Z}_{W''}\rangle$ (see Theorem 7.26 in \cite{ban3}), where $W$ is the result of gluing a cobordism $W'$ from $M$ to $N$ with a cobordism $W''$ from $N$ to $P$ along $N$.
Note that the preparatory results Proposition 7.23, Lemma 7.24, and Proposition 7.25 in \cite{ban3} need only be modified by replacing the \textbf{Br}-valued action functional $\mathbb{S}$ with the \textbf{cBr}-valued action functional $\overline{\mathbb{S}}$ in the formulation.
Diffeomorphism invariance $\varphi_{\ast}(\overline{Z}_{W}) = \overline{Z}_{W'}$ (see Theorem 9.16 in \cite{ban3}) of our state sum under diffeomorphisms $\varphi \colon \partial W \rightarrow \partial W'$ that can be extended to so-called time consistent diffeomorphisms $\varphi \colon W \rightarrow W'$
can be shown along the lines of Section 9 in \cite{ban3}.
In particular, Lemma 9.12 and Lemma 9.14 in \cite{ban3} remain valid when replacing $\mathbb{S}$ with $\overline{\mathbb{S}}$ in the formulation.
The map $\varphi_{\ast} \colon \overline{Z}(\partial W) \rightarrow \overline{Z}(\partial W')$ can then be defined on a function $z \colon \mathcal{F}(\partial W) \rightarrow \overline{Q}$ and a field $g \in \mathcal{F}(\partial W')$ by
$$\varphi_{\ast}(z)(g) = z(g \circ (\operatorname{id}_{[0, 1]} \times \varphi)) \in \overline{Q}.$$

\subsection{The aggregate invariant and exotic Kervaire spheres}\label{aggregate invariant}
Positive TFTs have been created by Banagl \cite{ban2} with the intention to provide new topological invariants for high-dimensional manifolds.
In this section, we explain how our positive TFT $\overline{Z}$ constructed in the previous section can be used to assign to every homotopy sphere $M$ its \emph{aggregate invariant} $\overline{\mathfrak{A}}(M)$, an element of the complete semiring $\overline{Q}$ from \Cref{idempotent profinite completion}.
The construction of $\overline{\mathfrak{A}}$ is analogous to that of the aggregate invariant $\mathfrak{A}$ studied Section 10 in \cite{ban3}.
While the invariant $\mathfrak{A}$ is known to distinguish exotic spheres from the standard sphere (see Corollary 10.4 in \cite{ban3}), we will indicate briefly that the invariant $\overline{\mathfrak{A}}$ can distinguish exotic Kervaire spheres from other exotic spheres in infinitely many dimensions.

Fix a closed $(n-1)$-manifold $M$ which is homeomorphic (but not necessarily diffeomorphic) to the sphere $S^{n-1}$.
Without loss of generality, we assume in the following that $S^{n-1}$ and $M$ are smoothly embedded in $\{0\} \times \mathbb{R}^{D-1}$ (compare \Cref{cobordisms}).
From now on, we suppose that $n-1 \geq 5$.
Then, it is well-known that $M$ admits Morse functions with exactly two non-degenerate critical points, namely one minimum and one maximum.
Given any diffeomorphism $\xi \colon [0, 1] \rightarrow [a, b]$ with $\xi(0) = a$, and any Morse function $f_{M} \colon M \rightarrow \mathbb{R}$ with exactly two non-degenerate critical points, we observe that the map
$$
\overline{f}_{M} \colon [0, 1] \times M \rightarrow \mathbb{R}^{2}, \qquad \overline{f}_{M}(t, x) = (\xi(t), f_{M}(x)),
$$
is a fold field on $M$ (see \Cref{fold field}).
Let $C_{2}(M) \subset \mathcal{F}(M)$ denote the (non-empty) subset of all such maps $\overline{f}_{M}$.
Fix an element $\overline{f}_{S} \in C_{2}(S^{n-1})$ of the form $\overline{f}_{S} = \operatorname{id}_{[0, 1]} \times f_{S}$.
Let us write $\operatorname{Cob}(S^{n-1}, M)$ for the collection of all oriented cobordisms from $S^{n-1}$ to $M$ that are embedded in $[0, 1] \times \{0\} \times \mathbb{R}^{D-1}$ (compare property (4) of \Cref{kobordismus}).
Since $M$ is homeomorphic to $S^{n-1}$, it can be shown that $\operatorname{Cob}(S^{n-1}, M)$ is non-empty.
(The proof, which exploits the fact that any such $M$ is parallelizable, is given in Lemma 10.1 in \cite{ban3}.)
Now, for any cobordism $W \in \operatorname{Cob}(S^{n-1}, M)$ and any fold field $\overline{f}_{M} \in C_{2}(M)$, the state sum $\overline{Z}_{W} \in \overline{Z}(S^{n-1}) \widehat{\otimes} \overline{Z}(M)$ of \Cref{state module} can be evaluated at $(\overline{f}_{S}, \overline{f}_{M}) \in \mathcal{F}(S^{n-1}) \times \mathcal{F}(M)$ to yield an element $\overline{Z}_{W}(\overline{f}_{S}, \overline{f}_{M})$ in the complete semiring $\overline{Q}$ from \Cref{idempotent profinite completion} that is associated to a faithful linear representation $Y \colon \mathbf{cBr} \rightarrow \mathbf{Vect}$.
Hence, summation in the complete semiring $\overline{Q}$ yields a well-defined element
$$
\overline{\mathfrak{A}}(M) := \sum_{\overline{f}_{M} \in C_{2}(M)} \sum_{W \in \operatorname{Cob}(S^{n-1}, M)} \overline{Z}_{W}(\overline{f}_{S}, \overline{f}_{M}) \in \overline{Q}.
$$

In conclusion, we outline an application to Kervaire spheres, which are a concrete family of homotopy spheres that can be obtained from a plumbing construction as follows (see \cite[p. 162]{mich}).
The unique Kervaire sphere $\Sigma_{K}^{n-1}$ of dimension $n-1 = 4r+1$ can be defined as the boundary of the parallelizable $(4r+2)$-manifold given by plumbing together two copies of the tangent disc bundle of $S^{2r+1}$.
According to the classification theorem of homotopy spheres (see Theorem 6.1 in \cite[pp. 123f]{luck}), as well as recent work of Hill-Hopkins-Ravenel \cite{hhr} on the Kervaire invariant one problem, it is known that $\Sigma_{K}^{n-1}$ is an exotic sphere, i.e., homeomorphic but not diffeomorphic to $S^{n-1}$, except when $n-1 \in \{5, 13, 29, 61, 125\}$.

Note that, according to Remark 6.3 in \cite{wra3}, there are infinitely many dimensions of the form $n-1 \equiv 13 \; (\operatorname{mod} 16)$ in which there exist exotic spheres that are not diffeomorphic to the Kervaire sphere $\Sigma^{n-1}_{K}$.
The following result shows that our aggregate invariant $\overline{\mathfrak{A}}$ can distinguish exotic Kervaire spheres from other exotic spheres in infinitely many dimensions.
We give a sketch of the proof by referring to the results of \cite{wra}.
A detailed proof is beyond the scope of this paper, and will appear elsewhere.

\begin{theorem}\label{theorem kervaire spheres}
Suppose that $n-1 \equiv 13 \; (\operatorname{mod} 16)$ and $n-1 \geq 237$.
Then, an exotic $(n-1)$-sphere $\Sigma^{n-1}$ is diffeomorphic to the Kervaire sphere if and only if $\overline{\mathfrak{A}}(\Sigma^{n-1}) = \overline{\mathfrak{A}}(\Sigma_{K}^{n-1})$.
\end{theorem}

\begin{proof}[Sketch of proof]
Recall from \Cref{idempotent profinite completion} that elements of $\overline{Q}$ are families of power series in $\mathbb{B}[[\mathfrak{q}]]$ which are indexed by the loop-free morphisms of $\mathbf{cBr}$.
It follows from the construction of the state sum $\overline{Z}_{W}$ (see \Cref{state module}) that non-trivial power series of the element $\overline{\mathfrak{A}}(M) \in \overline{Q}$ can only occur in the factor
$$
\overline{Q}(H_{([2], \underline{0}), ([2], \underline{0})}) = \mathbb{B}[[\mathfrak{q}]] \oplus \mathbb{B}[[\mathfrak{q}]] \oplus \mathbb{B}[[\mathfrak{q}]],
$$
where the three copies of $\mathbb{B}[[\mathfrak{q}]]$ correspond to the three possible loop-free morphisms $([2], \underline{0}) \rightarrow ([2], \underline{0})$ in \textbf{cBr}, namely $1_{([2], \underline{0})}$, $b_{(0), (0)}$, and $i_{(0)} \circ e_{(0)}$.
Let $\zeta(\Sigma^{n-1}) \in \mathbb{B}[[\mathfrak{q}]]$ denote the component of $\overline{\mathfrak{A}}(\Sigma^{n-1})$ that corresponds to the loop-free morphism $i_{(0)}\circ e_{(0)}$
Then, for every $\nu \in \underline{\mathbb{N}}$ the coefficient of $\mathfrak{q}^{\nu}$ in $\zeta(\Sigma^{n-1})$ is nonzero if and only if there exists a fold field $F \in \mathcal{F}(\overline{f}_{S}, \overline{f}_{\Sigma})$ such that $\overline{\mathbb{S}}(F) = \left(\bigotimes_{k=0}^{\infty} \lambda_{(k)}^{\nu_{k}}\right) \otimes (i_{(0)}\circ e_{(0)})$.
We choose $\nu$ such that $\nu_{j} = 1$ for $j = n/2$ and $\nu_{j} = 0$ for $j \neq n/2$.
Then, it follows from Corollary 10.1.4 and Theorem 3.4.9 in \cite{wra} that the coefficient of $\mathfrak{q}^{\nu}$ in $\zeta(\Sigma^{n-1})$ is $1$ whenever $\Sigma^{n-1}$ is diffeomorphic to $\Sigma^{n-1}_{K}$.
Conversely, if $\Sigma^{n-1}$ is not diffeomorphic to $\Sigma^{n-1}_{K}$, then Corollary 10.1.4 in \cite{wra} implies that the coefficient of $\mathfrak{q}^{\nu}$ in $\zeta(\Sigma^{n-1})$ is $0$.
\end{proof}

\subsection*{Acknowledgments}
The authors would like to thank the referee for invaluable comments that helped improving the paper.
The authors are grateful to Prof. Banagl for providing the initial motivation for this work.
The second author is a JSPS International Research Fellow (Postdoctoral Fellowships for Research in Japan (Standard)).
The second author was supported by JSPS KAKENHI Grant Number JP18F18752.
The second author was also partially supported by a scholarship of the German National Merit Foundation.

\bibliographystyle{amsplain}

\begin{thebibliography}{99}

\bibitem{ati} M.F. Atiyah, \emph{Topological quantum field theory}, Publ. Math. Inst. Hautes \'{E}tudes Sci. \textbf{68} (1988), 175--186.

\bibitem{ban1} M. Banagl, \emph{The Tensor Product of Function Semimodules}, Algebra Universalis \textbf{70} (2013), no. 3, 213--226.
\bibitem{ban2} M. Banagl, \emph{Positive topological quantum field theories}, Quantum Topology \textbf{6} (2015), no. 4, 609--706.
\bibitem{ban3} M. Banagl, \emph{High-dimensional topological field theory, positivity, and exotic smooth spheres}, technical report, Heidelberg University (2016).

\bibitem{birwen} J.S. Birman, H. Wenzl, \emph{Braids, link polynomials and a new algebra},
Trans. Amer. Math. Soc. \textbf{313} (1989), no. 1, 249--273.

\bibitem{bra} R. Brauer, \emph{On algebras which are connected with the semisimple continuous groups},
Annals of Math. \textbf{38} (1937), no. 4, 857--872.

\bibitem{eil} S. Eilenberg, \emph{Automata, languages, and machines}, Pure and Applied Mathematics, vol. A, Academic Press, 1974.

\bibitem{gol85} M. Goldstern, \emph{Vervollst\"{a}ndigung von Halbringen}, Diplomarbeit, Technische Universit\"{a}t Wien, 1985.

\bibitem{gol} M. Golubitsky, V. Guillemin, \emph{Stable mappings and their singularities}, Grad. Texts Math., vol. 14, Springer-Verlag, New York, Heidelberg, Berlin, 1973.

\bibitem{hhr} M.A. Hill, M.J. Hopkins, D.C. Ravenel, \emph{On the non-existence of elements of Kervaire invariant one}, Annals of Mathematics \textbf{184} (2016), no. 1, 1--262.

\bibitem{kar} G. Karner, \emph{On limits in complete semirings}, Semigroup Forum \textbf{45} (1992), 148--165.

\bibitem{kas} C. Kassel, \emph{Quantum Groups}, Graduate Texts in Mathematics, Springer New York (1995).

\bibitem{kat} Y. Katsov, \emph{Tensor products and injective envelopes of semimodules over additively regular semirings}, Algebra Colloquium \textbf{4} (1997), no. 2, 121--131.

\bibitem{kat2} Y. Katsov, \emph{Toward homological characterization of semirings: Serre's conjecture and Bass' perfectness in a semiring context}, Algebra Universalis \textbf{52} (2004), 197--214.

\bibitem{kel80} G.M. Kelly, M.L. Laplaza, \emph{Coherence for compact closed categories}, Journal of Pure and Applied Algebra \textbf{19} (1980), 193--213.

\bibitem{kro} D. Krob, \emph{Monoides et semi-anneaux continus}, Semigroup Forum \textbf{37} (1988), 59--78.

\bibitem{mich} H.B. Lawson, M.-L. Michelsohn, \emph{Spin Geometry}, Princeton Math. Series 38, Princeton University Press, (1989).

\bibitem{leh} G.I. Lehrer, R.B. Zhang, \emph{The Brauer category and invariant theory},
Journal of the European Mathematical Society \textbf{17} (2015), 2311--2351.

\bibitem{luck} W. L\"{u}ck, \emph{A basic introduction to surgery theory}, version: October 27, 2004, \url{http://131.220.77.52/lueck/data/ictp.pdf}.

\bibitem{mue} L.F. M\"{u}ller, \emph{Linear representations of the Brauer category}, Master thesis, Heidelberg University (2015).

\bibitem{mur} J. Murakami, \emph{The representations of the $q$-analogue of Brauer's centralizer algebras and the Kauffman polynomial of links},
Publ. Res. Inst. Math. Sci. \textbf{26} (1990), no. 6, 935--945.

\bibitem{parkam} M. Parvathi, M. Kamaraj, \emph{Signed Brauer's algebras},
Communications in Algebra \textbf{26} (1998), no. 3, 839--855.

\bibitem{sak} J. Sakarovitch, \emph{Kleene's theorem revisited}, Lect. Notes in Comp. Sci. \textbf{281} (1987), 39--50.

\bibitem{sch} P. Schauenburg, \emph{Turning monoidal categories into strict ones}, New York J. Math. \textbf{7} (2001), 257--265.

\bibitem{tur} V. Turaev, \emph{Operator invariants of tangles, and
$R$-matrices}, Math. USSR-Izvestiya 35 (1990), no.2, 411--444.

\bibitem{wen} H. Wenzl, \emph{Quantum groups and subfactors of type $B$, $C$, and $D$},
Comm. Math. Phys. \textbf{133} (1990), no. 2, 383--432.

\bibitem{wra} D.J. Wrazidlo, \emph{Fold maps and positive topological quantum field theories}, Dissertation, Heidelberg University (2017), \url{http://nbn-resolving.de/urn:nbn:de:bsz:16-heidok-232530}.

\bibitem{wra2}
D.J. Wrazidlo, \emph{Singular patterns of generic maps of surfaces with boundary into the plane},
in: Singularities --- Kagoshima 2017, Proceedings of the 5th Franco-Japanese-Vietnamese Symposium on Singularities (2020), 235--263. \url{https://doi.org/10.1142/9789811206030_0012}

\bibitem{wra3} D.J. Wrazidlo, \emph{Bordism of constrained Morse functions}, preprint (2018), arXiv: \url{https://arxiv.org/abs/1803.11177}.
\end{thebibliography}

\end{document}